\documentclass[a4paper,fleqn,10pt]{article}
\usepackage{fullpage,color}
\usepackage{dsfont}

\usepackage[T1]{fontenc}
\usepackage[utf8]{inputenc}
\usepackage[english]{babel}
\usepackage{amsmath}
\usepackage{amsfonts}
\usepackage{amssymb}
\usepackage{amsthm}
\usepackage{mathpazo}
\usepackage{eqnarray}
\usepackage{comment}
\usepackage[pdftex]{graphicx}
\usepackage[margin=0.7in]{geometry}
\usepackage{float}
\usepackage{cite}
\usepackage{authblk}

\usepackage{tikz}
\usetikzlibrary{patterns}
\usetikzlibrary{arrows,intersections,quotes,calc}
\usepackage{pgfplots}
\definecolor{totocolor}{rgb}{1,0.6,1}
\colorlet{LightRed}{red!30}
\colorlet{ColorPink}{blue!20}
\colorlet{lightgreen}{green!50}
\tikzstyle{terminal}=[circle,draw]
\pgfplotsset{width=7cm,compat=1.15}

\usepackage{accents}

\usepackage{comment}
\usepackage{hyperref}

\newcommand{\cqfd}{\nobreak \ifvmode \relax \else
      \ifdim\lastskip<1.5em \hskip-\lastskip
      \hskip1.5em plus0em minus0.5em \fi \nobreak
      \vrule height0.75em width0.5em depth0.25em\fi}

\makeatletter
\newcommand{\leqnomode}{\tagsleft@true}
\makeatother

\def\<{\langle}
\def\>{\rangle}

\def\Chi{\raise .3ex
\hbox{\large $\chi$}}

\def\({\Bigl (}
\def\){\Bigr )}

\newcommand{\bea}{$$ \begin{array}{lll}}
\newcommand{\eea}{\end{array} $$}
\newcommand{\bi}{\begin{itemize}}
\newcommand{\ei}{\end{itemize}}

\newtheorem{theorem}{Theorem}

\newtheorem{lemma}{Lemma}

\newtheorem{definition}{Definition}
\newtheorem{remark}{Remark}
\newtheorem{proposition}{Proposition}
\DeclareMathOperator{\R}{{\mathbb R}}

\newcommand{\aeq}{A_1^{\infty}}
\newcommand{\qeq}{Q_2^{\infty}}
\newcommand{\xeq}{Q_3^{\infty}}
\newcommand{\dt}{\frac{d}{dt}}
\renewenvironment{proof}{\noindent{\bf Proof.}}{\hfill
  $\blacksquare$\par\noindent}

\providecommand{\keywords}[1]
{
  \small	
  \textbf{\textit{Keywords: }} #1
}

\begin{document}
\title{Tumorigenesis and axons regulation for the pancreatic cancer: a mathematical approach.}
\date{}

\author[1]{\small Sophie Chauvet}
\author[2]{\small Florence Hubert}
\author[1]{\small Fanny Mann}
\author[2, 3]{\small Mathieu Mezache}

\affil[1]{\footnotesize Aix Marseille Univ, CNRS, IBDM, Marseille, France}
\affil[2]{\footnotesize Aix Marseille Univ, CNRS, Centrale Marseille, I2M, Marseille, France}
\affil[3]{\footnotesize Université Paris-Saclay, INRAE, MaIAGE, 78350, Jouy-en-Josas, France}

\maketitle
\begin{abstract}
{
The nervous system is today recognized to play an important role in the development of cancer. Indeed, neurons extend long processes (axons) that grow and infiltrate tumors in order to regulate the progression of the disease in a positive or negative way, depending on the type of neuron considered. Mathematical modelling of this biological process allows to formalize the nerve-tumor interactions and to test hypotheses in silico to better understand this phenomenon. In this work, we introduce a system of differential equations modelling the progression of pancreatic ductal adenocarcinoma (PDAC) coupled with associated changes in axonal innervation. The study of the asymptotic behavior of the model confirms the experimental observations that PDAC development is correlated with the type and densities of axons in the tissue. In addition, we study the identifiability of the model parameters. This informs on the adequacy between the parameters of the model and the experimental data. It leads to significant insights such that the transdifferentiation phenomenon accelerates during the development process of PDAC cells. Finally, we give an example of a  simulation of  the effects of partial or complete denervation that  sheds lights on complex
correlation between the cell populations and axons with opposite functions.

} 
\end{abstract}

\section{Introduction}
Pancreatic ductal adenocarcinoma (PDAC) is a leading cause of cancer death in men and women.  Late detection of this cancer, due to the near absence of symptoms in the early stages, is associated with a poor prognosis and an overall 5-year survival rate of less than 5$\%$ \cite{bengtsson2020actual}. In recent decades, the impact of the microenviroment on tumour progression has become widely recognized, which has led to  the development of new therapies such as immunotherapies.  More recently, it has been shown that fibers of the nervous system infiltrate the tumor microenvironment where they also participate in the regulation of cancer development and progression \cite{guillot2020sympathetic}. It is therefore relevant to study the neurobiology of cancers through mathematical modelling in order to simulate the responses of cancer cells to innervation and predict at long term the effects of therapies targeting neuron-tumor interactions. \\

Mathematical modelling of the impact of the microenvironment on tumor progression has been widely investigated for many types of cancers. In the case of pancreatic cancer, a model of the interplay between the immune system and tumor progression has been proposed \cite{louzoun2014}. As far as we know, mathematical modeling of the neural regulation of tumor progression has only be performed for prostate cancer \cite{lolas2016tumour}. This model confirmed experimental observations that a tumor is able to recruit nerves that, in turn, promote tumor development and metastatic spread. However, this initial model did not take into account the full functional diversity of neurons of the peripheral nervous system (PNS) and in particular their potential tumor suppressive effect discovered more recently in PDAC. To our knowledge, no mathematical model integrating the antitumor and protumor activities of the PNS currently exists.  \\

In this article, we developed an ordinary differential equation (ODE) model that describes and simulates the relationship between the PNS and pancreatic cancer development. The model is based on and calibrated with experimental data obtained from a genetically engineered mouse model of PDAC, in which the innervation of early pre-cancerous lesions and cancer have been characterized by three-dimensional (3D) histology  (c.f \cite{guillot2020sympathetic}). This model aims to investigate how dynamic changes in the neuronal composition of the microenvironment influence tumor progression.
\\

{The paper is organized as follows. In Section \ref{sec:modeling}, we review the biological background behind the mechanisms of the PDAC progression. We introduce the mathematical model and detail the assumptions made. 
In Section \ref{sec:model}, we study the mathematical properties of the model. We prove its well-posedness and the convergence towards the pathological equilibrium under some assumptions. We extract some exponential convergence estimators which allow us to reduce the system and performed the asymptotic analysis on the limit system.
In Section \ref{sec:calib}, we study the identifiability of the parameters when the model is confronted to the experimental data. We perform a sensitivity analysis which sheds lights on the effect of the axons on the PDAC progression.}
\section{Modeling the evolution of cell populations and axons}\label{sec:modeling}
\subsection{Biological background}

The PNS is a vast network of nerves and ganglia that connect the brain to the other organs of the body. It consists of both afferent (sensory) nerves and efferent (motor) nerves that carry information in and out of the brain, respectively. While essential for internal body communication and proper regulation of physiological functions, the PNS also plays a newly-identified and pivotal role in the control of tumorigenesis. For example, denervation experiments in animal models of prostate and gastric cancers demonstrated a role of the visceral efferent motor system (also known as the autonomic nervous system) in promoting tumor progression and metastasis \cite{magnon2013autonomic,zhao2014denervation}. 
These findings have led to the emerging concept of “nerve dependence in tumorigenesis” and a growing interest in repositioning inhibitors of nerve signaling for cancer treatment \cite{boilly2017nerve,zahalka2020nerves}.\\

\textbf{The bifunctionnal role of the PNS in pancreatic tumorigenesis.} The impact of the PNS varies considerably depending on the tumor site. This has been highlighted by studies in animal models of PDAC. 
Indeed, and in contrast to its promoting role in prostate cancer, the autonomic nervous system has appeared to exert tumor suppressive effects in PDAC. The autonomic nervous system is divided functionally and anatomically into the sympathetic and parasympathetic nervous systems, which work together synergistically to regulate pancreatic functions \cite{love2007autonomic}. 
In PDAC, several studies reported that transection of the vagus nerve, which provides parasympathetic inputs to the pancreas, promotes pancreatic cancer progression \cite{partecke2017subdiaphragmatic,renz2018cholinergic}. 
A similar acceleration of PDAC development and increased metastasis have been reported after selective depletion of pancreatic sympathetic innervation \cite{guillot2020sympathetic}, further supporting a protective function of the autonomic nervous system in this type of cancer. Conversely, the pro-tumoral influence of the PNS on PDAC is exerted by sensory neurons, whose selective ablation or functional silencing slows tumor progression and improves survival \cite{bai2011inhibition,saloman2016ablation,sinha2017panin}.
Finally, when both sympathetic and sensory innervation of the pancreas are diminished \cite{guillot2020sympathetic}, this leads to an acceleration of PDAC development, suggesting a preponderant influence of the autonomic nervous system during tumor initiation. Thus, taking into account the functional specializations of nerve subtypes and their integration is crucial for understanding and predicting the trajectory of innervated tumors.\\

\textbf{Neuroplastic changes associated with tumorigenesis.} The PNS control over tumorigenesis is based on its ability to innervate developing tumors and release neurotransmitters, or other factors, in the cancer cell environment. 
Precise mapping of pancreatic tissue innervation and its evolution during PDAC development has been performed on histological sections and more recently using 3D light-sheet fluorescence microscopy (LSFM) on murine and human pancreas \cite{chien2019human, makhmutova2021optical}.
The results revealed striking differences in the autonomic and sensory innervation patterns of the healthy pancreas, with a dense meshwork of autonomic nerve fibers (both sympathetic and parasympathetic) throughout the exocrine pancreas, from which PDAC arises, and an absence of sensory fibers in these same regions, the latter residing along the arteries and innervating the pancreatic islets \cite{guillot2020sympathetic,lindsay2006quantitative}.

An initiating event for the development of PDAC is the trans-differentiation of acinar cells (the functionnal unit of the exocrine pancreas) into progenitor-like cells with ductal characteristics, a process called acinar-to-ductal metaplasia (ADM).  ADM can progress to form premalignant pancreatic intraepithelial neoplasia (PanIN) and eventually pancreatic cancer \cite{storz2017acinar}. 
A substantial innervation of the early pancreatic lesions by autonomic axons has been reported, with PanINs appearing as hotspots of sympathetic hyperinnervation \cite{guillot2020sympathetic}. 
While some sensory fibers can also be detected around PanINs, their density remains relatively low compared to autonomic fibers \cite{sinha2017panin,stopczynski2014neuroplastic}. 
In invasive PDAC tumors, however, this picture is completely reversed: a high density of sensory fibers deeply infiltrates the center of the tumors, while a moderate sympathetic innervation limited to the peripheral regions of the tumors was reported \cite{ceyhan2009pancreatic,guillot2020sympathetic}. In conclusion, the data revealed stage-specific remodeling of PNS networks during tumorigenesis that may have an important function by shifting an initially protective neural environment (autonomic $>$ sensory) into a milieu favorable for cancer cell growth (sensory $>$ autonomic). \\

\subsection{Mathematical model}
We will focus on the pancreas as the main domain of our model, thus including both the pancreatic cells and the neighboring nerve fibers (or axons). In this model, we distinguish between the cell concentrations denoted $Q_i$ for $i\in \{0,1,2,3\}$, with:  \textit{Acini} $Q_0$, \textit{ADM} $Q_1$, \textit{PanIN} $Q_2$ and \textit{PDAC} $Q_3$. The variables corresponding to PNS axons are denoted $A_1, \; A_2$.  $A_1$ is the variation of the density of \textit{autonomic axons} with respect to its equilibrium at initial state (denoted $A_1^{eq}$). Hence, we consider the following
$$ A_1 \quad = \quad \text{density of autonomic axons} - \text{initial equilibrium density}.$$
This formalism allows A1 to take negative values. One can justify it by the fact that neuroplastic changes of autonomic axons are non-monotonous: the density increases in PanIN and decreases in PDAC compared to the initial equilibrium density, ie., the density in acini \cite{ceyhan2009pancreatic,guillot2020sympathetic}. Finally,   $A_2$ is the density of \textit{sensory axons}.\\

\begin{figure}[ht!]
\begin{tikzpicture}
\tikzstyle{pop}=[fill=black!10]
\tikzstyle{axo}=[fill=red!10]
\tikzstyle{edge}=[->,>=latex,ultra thick]
\tikzstyle{mineff} =[-|,>=latex,dashdotted, red!50, ultra thick]

\node[pop] (H) at (-2,-2) {Acinus $Q_0$};
\node[pop] (C1) at (2,-2) {ADM $Q_1$};
\node[pop] (C2) at (6,-2) {PanIN $Q_2$};
\node[pop] (C3) at (10,-2) {PDAC $Q_3$};
\node[axo] (A1) at (2,1) {autonomic axons $A_1$};
\node[axo] (A2) at (6,-5) {Sensory axons $A_2$};
\draw[edge] (H) -- (C1);
\draw[edge] (C1) -- (C2);
\draw[edge] (C2) -- (C3);
\draw[->,>=latex,dashed, thick] (C1) -- (A1) ;
\draw[->,>=latex,dashed] (C2) -- (A1) ;
\draw[-|,>=latex,dashdotted, ultra thick] (C3) -- (A1.east) ;
\draw[->,>=latex,dashed] (C3) -- (A2) ;
\draw[->,>=latex,dashed] (C2) -- (A2) ;
\draw[mineff] (A1) --(3.5,-1.9);
\draw[mineff] (A1) --(7.5,-1.9);
\draw[->,>=latex,dashed, red!50] (A2) --(7.5,-2.1);
\draw[->,>=latex, red!50] (A2) --(8.3,-2.3);
\draw[->,>=latex, red!50] (A2) --(4.4,-2.3);
\draw[->,>=latex, red!80] (A1) --(4.4,-1.7);
\draw[->,>=latex, red!80] (A1) --(8.3,-1.7);
\draw[-latex,bend left]  (C2) edge (0,-2);
\draw[-latex,bend left]  (C3) edge (0,-2);
\path (C2) edge [loop left] (C2);
\path (C3) edge [loop left] (C3);
\end{tikzpicture}

\caption{\textbf{Schematic representation of the interactions among the model variables.} {\small Each variable corresponds to a rectangular box; note that cell populations are in gray while axons are in red. {Solid thin black arrows denote regulation of either proliferation of cells or transfer rates and solid thick black arrows denote enhancement of the transfer rate.} Dashed black arrows denote enhancement of the growth, dashed black lines with a vertical end denote inhibition of the growth. Solid red arrows denote enhancement of the proliferation, dashed red arrows denote enhancement of the transfer rate and dashed red lines with a vertical end denote inhibition of the transfer rate.}\label{fig:reac_scheme}}
\end{figure}
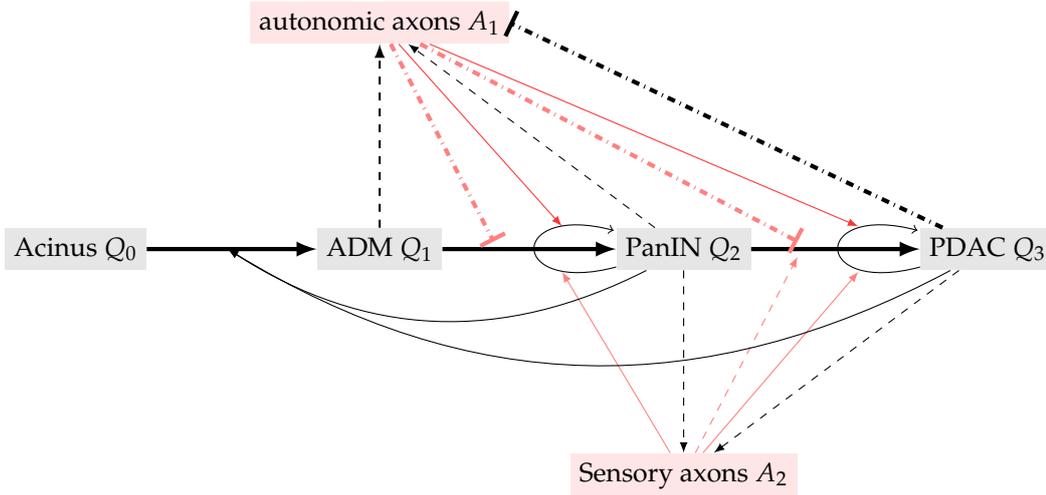

We propose multi-compartmental model in which the growth of the PNS is coupled to the transfer from healthy acini to PDAC, and to the proliferation of pre-cancerous (PanIN) and cancerous (PDAC) cells. A schematic of the model, showing the variables and their interactions, can be found in Figure \ref{fig:reac_scheme}.\\

\textbf{Transfer between compartments.} We consider that Acini progress in ADM with a transfer rate $\pi_0$. The transfer term is up-regulated by the presence of PanIN and PDAC through a Michaelis-Menten term with maximal amplitude $\delta_0$ and 1 as the Michaelis constant. This assumption leads to the following interpretation: the PanIN and the PDAC are able to self-promote and consequently decrease the concentration of Acini in the system. However the increasing influence of PanIN and PDAC on the transfer term is saturated. This transfer drives the dynamics of Acini and is described by the following equation:
\begin{align}
\begin{split}
\frac{d}{dt} Q_0(t)= &\underbrace{-\pi_0\left[1 + \delta_0 \frac{Q_2(t) +Q_3(t)}{1+Q_2(t) +Q_3(t)}\right] Q_0(t).}_\text{\parbox{6cm}{transfer from Acini to ADM up regulated by the PanIN and PDAC}} \label{eq:Mod_Q0}
\end{split}
\end{align}
Next, ADM become PanIN with a transfer rate $\pi_1$. Moreover, the appearance of PanIN is positively correlated with a high concentration of autonomic axons that have an inhibitory effect on cancer progression. We model this phenomenon by adding the dependency of amplitude $-\beta_1$ on $A_1$ to the transfer term. Moreover, we add a multiplicative regularization term $\rho(A_1)$ which ensures that if $A_1$ takes negative values, $\rho(A_1)$ is almost equal to 0 e.g. $$\rho(x)=\tfrac{1}{2}\left(1+\tfrac{x}{\sqrt{x^2 +\epsilon}} \right)\;, \; 0<\epsilon\ll 1.$$ It implies that if the density of autonomic axons in the system is low, the regulatory impact of axons is also low or non-existent. Hence, the dynamics of ADM is described by the following equation:
\begin{align}
\begin{split}
\frac{d}{dt} Q_1(t) = &\pi_{0}\left[1 + \delta_0\frac{Q_2(t) +Q_3(t)}{1+Q_2(t) +Q_3(t)}\right] Q_0(t) -\underbrace{\pi_1 \left[1-\beta_1A_1(t) \rho\left(A_1(t)\right)\right] Q_1(t)}_\text{\parbox{4cm}{transfer from ADM to PanIN down regulated by the autonomic axons}}.
\end{split}
\end{align}
Finally, PanIN progress to PDAC with a transfer rate $\pi_2$. This process coincides with the increase of sensory axons that enhance PDAC growth and a decrease of  autonomic axons that have the opposite effect. We model this mixed effect in the transfer term by adding a linear dependencies of amplitude $\delta_2$ on $A_2$ and of amplitude $-\beta_2$ on $A_1$. {In addition, we still introduce the regularization term on the inhibiting effect of the autonomic axons which ensures that there is no effect on the transfer in case the density of autonomic axons $A_1$ is negative. It leads to the following mathematical formulation of the dynamics transfert:}
\begin{align*}
\begin{split}
\frac{d}{dt}Q_2(t) = & \textit{ proliferation term} +\pi_1 \left[1-\beta_1A_1(t) \rho\left(A_1(t)\right)\right] Q_1(t)-\underbrace{\pi_{2} \left[1-\beta_2A_1(t) \rho\left(A_1(t)\right)+\delta_2 A_2(t) \right] Q_2(t).}_\text{\parbox{6.5cm}{transfer from PanIN to PDAC down regulated by the autonomic axons and up regulated by the sensory axons}} 
\end{split}
\end{align*}
\\

\textbf{Proliferation terms.} We model cell proliferation in PanIN and PDAC by adding a logistic-like growth term for $Q_2$ and $Q_3$ of respective rate $\gamma_2$ and $\gamma_3$. The saturation term $\tau_C$ is applied to the total concentration of proliferating cells which corresponds to $Q_2 + Q_3$. We model the PNS effect in the growth process by incorporating the axons in the logistic law. We assume that the sensory axons promote the self-renewing growth of (pre)cancerous cells until the population attains a certain threshold. Furthermore, we assume that the autonomic axons have a mixed effect on this growth term. When the autonomic axons concentration is above its initial equilibrium state $A_1^{eq}$, the growth of (pre)cancerous cells is promoted. Once the concentration is lower than its equilibrium state, the cancerous cells still proliferate but attain a lower {carrying capacity}. The effects of $A_1$ and $A_2$ in the logistic law can be interpreted as tumor growth factors or inhibition of tumor growth factors and these effects are limited by the thresholds $\tau_{A_1}^C$ and $\tau_{A_2}^C$.    
The dynamics of PanIN and PDAC are then described by the following equations:
\begin{align}
\begin{split}\label{eq:Mod_Q2}
\frac{d}{dt} Q_2(t) = &\underbrace{\gamma_{2}Q_2(t)\left(1 -\tfrac{Q_2(t)+Q_3(t)}{\tau_{C}}+ \tfrac{A_1(t)}{\tau_{A_1}^C}+\tfrac{A_2(t)}{\tau_{A_2}^C}\right)}_\text{\parbox{4.5cm}{proliferation regulated by the axons}} +\pi_1 \left[1-\beta_1A_1(t) \rho\left(A_1(t)\right)\right] Q_1(t)\\
&-\pi_{2} \left[1-\beta_2A_1(t) \rho\left(A_1(t)\right)+\delta_2 A_2(t) \right] Q_2(t)
\end{split}
\\[2ex]
\begin{split}\label{eq:Mod_Q3}
\frac{d}{dt} Q_3(t) = & \underbrace{\gamma_{3}Q_3(t)\left(1 -\tfrac{Q_2(t)+Q_3(t)}{\tau_{C}}+\tfrac{A_1(t)}{\tau_{A_1}^C}+\tfrac{A_2(t)}{\tau_{A_2}^C}\right)}_\text{{proliferation regulated by the axons}} + \pi_2 \left[1-\beta_2A_1(t) \rho\left(A_1(t)\right)+\delta_2 A_2(t) \right] Q_2(t)
\end{split}
\end{align}\\
\textbf{{Axon growth dynamics.}} Neuroplastic changes occur during the tumorigenesis and are closely linked to the presence of pre-cancerous and cancerous cells. Modeling innervation with a logistic law seems the natural way to describe this phenomenon if no spatial representation is taken into account. However, we do not take constant growth rates since the innervation is clearly induced by precancerous and cancerous cells.\\ 
We assume from the experimental data that the growth rate of autonomic axons is increased by ADM and PanIN, whereas PDAC has an opposite effect. The coefficients $\alpha_1$, $\alpha_2$ and $ \alpha_3$ are associated respectively to $Q_1,$ $Q_2$ and $Q_3$ in the growth term in order to specify the effect of each cells on the innervation. Also, the PanIN and PDAC cells promote sensory axon growth in a similar way with coefficients $\bar{\alpha}_2 $ and {$\bar{\alpha}_3 $}. One can reasonably consider that growth in a biological phenomenon is saturated because of various biophysical constraints. We introduced the threshold $\tau_{A_2}$ which is an upper bound for the density of sensory axons in the system. Similarly, we denote $\tau_{A_1}$ the threshold on autonomic axons. However, $A_1$ is not a density but a difference quantity and this quantity is non-monotonous throughout the PDAC development process. Hence, we model the dynamic of $A_1$ by a modified logistic growth where the variable is bounded in $\left[-\tau_{A_1},\;\tau_{A_1}\right]$. If its growth term is non-negative, $A_1$ is tending to its upper bound and if the reverse is true, $A_1$ tends to its lower bound. To conclude, the growth dynamics of the axons are describe by the following equations
\begin{align}
\begin{split}
\frac{d}{dt} A_1(t) = &\underbrace{\left(\alpha_{1}Q_1(t)+\alpha_2 Q_2(t)-\alpha_3 Q_3(t) \right)}_\text{\parbox{4.5cm}{stimulus effect from ADM and PanIN and inhibiting effect from PDAC}} \underbrace{\left(1+\tfrac{A_1(t)}{\tau_{A_1}} \right)\left(1-\tfrac{A_1(t)}{\tau_{A_1}} \right)}_\text{logistic-like growth}
\end{split}
\\[2ex]
\begin{split}
\frac{d}{dt} A_2(t)  = &\underbrace{\left(\bar{\alpha}_2 Q_2(t) +\bar{\alpha}_3 Q_3(t)\right)}_\text{\parbox{4.5cm}{stimulus effect from PanIN and PDAC}} \underbrace{A_2(t)\left(1 - \tfrac{A_2(t)}{\tau_{A_2}} \right)}_\text{logistic growth}\label{eq:Mod_A2}
\end{split}
\end{align}

\textbf{A priori conditions and assumptions on parameters.} The interaction between cells and axons and the transition between cell populations are modeled by a dynamical system driven by the set of parameters $\left\lbrace \pi,\; \delta,\; \beta,\; \tau,\; \gamma,\; \alpha \right\rbrace $. In order to sum up, we recall that the \textit{transition rates} are denoted by $\pi$, the \textit{saturation rates}  by $\tau$ , the \textit{growth rates} by $\gamma$. The parameters $\beta$ and $\delta$ appear in the transition terms. The $\beta$ parameters are coefficients which translate the inhibiting effect on the transition rates whereas the $\delta$ parameters translate a stimulating effect on the transition rates. The last category of parameters is assimilated to the \textit{growth term} of the axons. The parameters $\alpha$ and $\bar{\alpha}$ are associated to the impact of the cell populations on the growth rate of axons. All these parameters are assumed to be non-negative. In the following, we assume the hypotheses :\\

\textbf{Hypothesis 1 :} the transfer terms cannot become negative, it implies the following sufficient conditions : 
\begin{equation}\label{hyp:first}
\tag{H1}
    1>\beta_1 \tau_{A_1}, \quad  1>\beta_2\tau_{A_1}.
\end{equation}

\textbf{Hypothesis 2 :} the growth rate of the cells in the logistic law gives rise to a competition between the PanIN cells and the PDAC cells for the same resource. However, the PDAC cells are considered to be dominant in the system and in the reality. It is translated by the following order relation on the parameters : 
\begin{equation}\label{hyp:second}
\tag{H2}
    \gamma_2 <\gamma_3 .
\end{equation}

\textbf{Hypothesis 3 :} the autonomic axons $A_1$ have a mixed effect on the proliferation term of PanIN and PDAC: it can either increase or decrease the resource in the logistic law. It is therefore unrealistic to consider that this mixed effect is the one that governs the dynamics of growth. This implies that the proliferation term cannot be of negative sign in the equations \eqref{eq:Mod_Q2} and \eqref{eq:Mod_Q3}. The following assumptions is assumed :
\begin{equation}\label{hyp:third}
\tag{H3}
   \tau_{A_1} < \tau_{A_1}^C.
\end{equation}
\medskip

\textbf{Initial conditions.}  We assume that at time 0, there are only Acini, autonomic axons and a very little amount of sensory axons. Hence, we have the following initial conditions :
\begin{equation}
Q_0(0) >0,\quad Q_1(0)=Q_2(0)=Q_3(0)=0,\quad A_1(0)=0,\quad 0<A_2(0) \ll 1.
\label{eq:init_cond}
\end{equation}
\medskip

Equations \eqref{eq:Mod_Q0}-\eqref{eq:Mod_A2} form a non-linear dynamical system. The mathematical analysis of the system, such as the well-posedness, the positivity and the long-term behavior, assesses  theoretically the legitimacy of modeling choices and improves the understanding of the interaction between axons and cancer.

\section{Properties of the model: well-posedness and asymptotic behavior.}\label{sec:model}
For the sake of clarity, we introduce an abstract formulation of the equations \eqref{eq:Mod_Q0}-\eqref{eq:Mod_A2}. The main change is that the transition terms are now represented by the nonnegative functions $f_i$ for $i=0,\;1,\; 2$. Hence, we obtain the following system :
\begin{equation}\label{eq:sys_abs}
\left\lbrace
\begin{array}{c@{}l}
\frac{d}{dt} Q_0= &-f_0\left(Q_2,Q_3\right) Q_0 \\
\\
\frac{d}{dt} Q_1 = &f_0\left(Q_2,Q_3\right) Q_0 -f_1\left(A_1 \right)Q_1\\
\\
\frac{d}{dt} Q_2 = &\gamma_{2}Q_2\left(1 -\tfrac{Q_2+Q_3}{\tau_{C}}+ \tfrac{A_1}{\tau_{A_1}^C}+\tfrac{A_2}{\tau_{A_2}^C}\right) +f_1\left(A_1\right) Q_1-f_2\left(A_1,A_2\right) Q_2\\
\\
\frac{d}{dt} Q_3 = &\gamma_{3} Q_3\left(1 -\tfrac{Q_2+Q_3}{\tau_{C}}+\tfrac{A_1}{\tau_{A_1}^C}+\tfrac{A_2}{\tau_{A_2}^C}\right) + f_2\left(A_1,A_2\right) Q_2\\
\\
\frac{d}{dt} A_1 = &\left(\alpha_{1}Q_1+\alpha_2 Q_2-\alpha_3 Q_3 \right)\left(1+\tfrac{A_1}{\tau_{A_1}} \right)\left(1-\tfrac{A_1}{\tau_{A_1}} \right)\\
\\
\frac{d}{dt} A_2  = &(\bar{\alpha}_2 Q_2 +\bar{\alpha}_3 Q_3)A_2\left(1 - \tfrac{A_2}{\tau_{A_2}} \right)
\end{array}
\right.
\end{equation}
The first step in understanding the interactions between axons and the cancer progression is to establish properties of the model such as  positiveness, well-posedness, etc.  and to study its asymptotic behavior.

\subsection{Properties of the system \eqref{eq:sys_abs}}
In the following section, we establish preliminary results on the model. Besides the well-posedness and the global existence of the solution, the cells populations are nonnegative. Moreover, the healthy cells populations (Acini and ADM) vanish at equilibrium. These results comfort the modeling since the simplistic but realistic outcome of the pancreatic cancer is the proliferation PDAC cells in the pancreas to the detriment of healthy cells.  
\begin{proposition}
\label{prop_pos}
Let $T>0$ be arbitrary and consider $X=\left(Q_0, Q_1, Q_2, Q_3, A_1, A_2\right)$ and $X^0 =\left(Q_0(0), Q_1(0), Q_2(0), Q_3(0), A_1(0), A_2(0)\right)$ such that $$Q_i(0) \geq 0,\quad \text{for}\quad i=0,1,2,3$$ and $$-\tau_{A_1} \leq A_1(0) \leq \tau_{A_1}\quad \text{and}\quad 0\leq A_2(0) \leq \tau_{A_2}.$$ 

Moreover assume that 
$$ \forall (x,y) \in \mathbb{R}^2\quad \exists \; M_0>m_0> 0 \implies m_0 \leq f_0(x,y) \leq M_0,$$
$$ \exists L_0 >0 \quad \text{ such that } \quad \forall u\in \mathbb{R}^2,\; \forall v\in \mathbb{R}^2 \quad \| f_0(u)- f_0(v)\| \leq L_0 \|u -v\| ,$$
and
$$ \forall x \in \mathbb{R} \quad \exists \; M_1>m_1> 0 \implies m_1 \leq f_1(x) \leq M_1,$$
$$ \exists L_1 >0 \quad \text{ such that } \quad \forall (u,v) \in \mathbb{R}^2 \quad \| f_1(u)- f_1(v)\| \leq L_1 \|u -v\| ,$$
and
$$ \forall (x,y) \in \mathbb{R}^2\quad  \exists \; M_2>m_2> 0 \implies m_2 \leq f_2(x,y) \leq M_2,$$
$$ \exists L_2 >0 \quad \text{ such that } \quad \forall u\in \mathbb{R}^2,\; \forall v\in \mathbb{R}^2  \quad \| f_2(u)- f_2(v)\| \leq L_2 \|u -v\|.$$

Then there exists a unique maximal solution for the system \eqref{eq:sys_abs} with the initial condition $X^0$ on $I=[0, T)$.
The following properties holds:\\
\begin{itemize}
\item \textbf{Nonnegativity} $\qquad \forall t\in I \quad Q_i(t) \geq 0,\quad \text{for}\quad i=0,1,2,3,$
\item \textbf{Boundedness of the axons}  $\qquad \forall t\in I \quad -\tau_{A_1} \leq A_1(t) \leq \tau_{A_1}\quad \text{and}\quad 0\leq A_2(t) \leq \tau_{A_2}.$
\end{itemize}

\end{proposition}
\begin{proof}

 The well-posedness is a direct consequence of the Cauchy-Lipschitz theorem and the Lipschitz bound on $f_i$. Moreover,
 $$\frac{d}{dt} Q_0(t) = -f_0(Q_2(t), Q_3(t))Q_0(t) = -\psi(t)Q_0(t) $$
 for a positive function $\psi$, then $\lbrace 0 \rbrace$ is an invariant set for $\tfrac{d}{dt}Q_0$ and we have $Q_0(t) \geq 0$.
 Similarly, we obtain that $$\forall t\in I \quad -\tau_{A_1} \leq A_1(t) \leq \tau_{A_1} \quad \text{and} \quad 0\leq A_2(t) \leq \tau_{A_2}.$$
Then, using the boundedness of $f_0$ and $f_1$ and the nonnegativity of $Q_0$, we obtain $$\dt Q_1(t) \geq - M_1 Q_1(t).$$
Hence by Gronwall's lemma, we have $$\forall t>0 \quad 0 \leq Q_1(0) e^{-M_1 t}\leq Q_1(t). $$
Similarly, using the boundedness of $f_1$ and the nonnegativity of $Q_1$, we obtain 
$$\forall t>0, \quad 0 \leq Q_2(0) e^{\bar{\psi}(t)} \leq Q_2(t) ,$$
where $$\bar{\psi}(t) = \int_0^t  \gamma_2 \left(1 -\tfrac{Q_2+Q_3}{\tau_{C}}+ \tfrac{A_1}{\tau_{A_1}^C}+\tfrac{A_2}{\tau_{A_2}^C}\right)-M_2ds. $$
Again, using the boundedness of $f_2$ and the nonnegativity of $Q_2$, we obtain 
$$\forall t>0, \quad 0 \leq Q_3(0) e^{\tilde{\psi}(t)} \leq Q_3(t) ,$$
where $$\tilde{\psi}(t) = \int_0^t  \gamma_3 \left(1 -\tfrac{Q_2+Q_3}{\tau_{C}}+ \tfrac{A_1}{\tau_{A_1}^C}+\tfrac{A_2}{\tau_{A_2}^C}\right)ds. $$
\end{proof}

\begin{proposition}\label{prop:exp_decay}
Suppose the same hypotheses as in Proposition \ref{prop_pos}, then the following properties holds:
\begin{enumerate}
\item \textbf{Exponential decays} The solutions $Q_0$ and $Q_1$ decay exponentially fast toward $0$ as $t$ tends to infinity. Moreover,
$$ \forall t\in I \quad Q_0(0) e^{-M_0 t} \leq Q_0(t) \leq Q_0(0) e^{-m_0 t},$$
and it exists a constant $C(Q_0(0), Q_1(0))\geq 0$ such that 
$$\forall t\in I \quad 0\leq Q_1(t) \leq C(Q_0(0), Q_1(0))e^{-(m_1 -o(1))t}.$$
\item \textbf{Global existence} There exists a unique solution for the system \eqref{eq:sys_abs} with initial condition $X^0$ on $\R_+$
\end{enumerate}
\end{proposition}
\medskip

\begin{proof}

1.\textit{(Exponential decays)} Using the Gronwall's lemma, we get the following from the boudedness of $f_0$ 
$$ \forall t>0 \quad Q_0(0) e^{-M_0 t} \leq Q_0(t) \leq Q_0(0) e^{-m_0 t}.$$
From the proof of nonnegativity, we obtain a lower bound on $Q_1$ which decays exponentially fast 
$$\forall t>0\quad  Q_1(0) e^{-M_1 t}\leq Q_1(t). $$
We are now interested in the upper bound of $Q_1$. Using the bound on $f_1$ and the Gronwall's lemma, we obtain 
\begin{align*}
\dt Q_1 (t) & \leq { M_0} e^{-m_0 t}Q_0(0) -m_1 Q_1(t),\\
Q_1(t) & \leq e^{-m_1 t}Q_1(0) + e^{-m_1 t} M_0 Q_0(0) \int_0^t e^{(m_1 -m_0)s} ds.
\end{align*}
We first consider the case where $m_1 = m_0$, thus we get 
\begin{align*}
Q_1(t) & \leq e^{-m_1 t}Q_1(0) + e^{-m_1 t} M_0 Q_0(0) t,\\
Q_1(t) & \leq e^{-m_1 t}Q_1(0) + e^{-m_1 t + ln(t)} M_0 Q_0(0),\\
Q_1(t) & \leq \left(e^{-ln(t)}Q_1(0) +  M_0 Q_0(0) \right)e^{-m_1 t + ln(t)}.\\
\end{align*}
We finally get 
$$ Q_1(t)\leq \big[Q_1(0)o(1) + M_0 Q_0(0) \big]e^{(-m_1 +o(1))t}.$$
Let us now consider the case where $m_1 \ne m_0$. 
\begin{align*}
Q_1(t) & \leq e^{-m_1 t}Q_1(0) + e^{-m_1 t} M_0 Q_0(0) \frac{e^{(m_1 - m_0)t -1}}{m_1-m_0},\\
Q_1(t) & \leq e^{-m_1 t}Q_1(0) + M_0 Q_0(0)\frac{e^{-m_0 t}- e^{-m_1 t}}{m_1 -m_0}.\\
\end{align*}
We denote $ \underline{m} = \min (m_0,m_1)$, we also obtain an exponential decay :
$$ Q_1(t) \leq \left[e^{-(m_1 -\underline{m})  t}Q_1(0) + M_0 Q_0(0)\frac{e^{-(m_0 -\underline{m}) t}- e^{-(m_1 -\underline{m}) t}}{m_1 -m_0}\right] e^{-\underline{m}t}.$$
\medskip 

2. \textit{(Global existence)} Now, we prove that $Q_2$ and $Q_3$ are bounded for $t\in I$.
Using the bounds on $A_1, \; A_2,\; f_1, \; f_2$, the positivity of $Q_2,\; Q_3$ and the fact that $Q_1$ is bounded by an exponential function, we get 
\begin{align*}
\dt Q_2(t) & \leq \gamma_2 Q_2(t) \left( C - \tfrac{Q_2(t)}{\tau_C}\right) + M_1 e^{-\underline{m}t} ,
\end{align*} 
where $C = 1 + \tfrac{\tau_{A_1}}{\tau_{A_1}^C} + \tfrac{\tau_{A_2}}{\tau_{A_2}^C}.$ Moreover, the function $f : x \mapsto ax(b-x)$ with $a>0$ and $b\in \mathbb{R}$ is uniformly bounded from above for $x\in \mathbb{R}^{+}$, hence we obtain that $Q_2$ is bounded for $t\in I$. Similarly, using the bounds on $A_1, \; A_2, \; f_2$, the positivity of $Q_2,\; Q_3$ and the fact that $Q_2$ is bounded by a constant denoted $C_2(T)$ on $[0,T]$, we get 
\begin{align*}
\dt Q_3(t)\leq \gamma_3 Q_3(t) \left( C - \tfrac{ Q_3(t)}{\tau_C} \right) +M_2 C_2(T),
\end{align*}
and that $Q_3$ is bounded for $t\in I$. Then the solution is global. 
\end{proof}

\subsection{Limit system and asymptotic behavior}
\subsubsection{Exponential convergence of the sensory axons}
As stated in Proposition \ref{prop_pos}, the Acini and ADM populations are reduced exponentially. It implies that once the proliferating PanIN and PDAC cells appears, the competition between the cells populations is in favour of the development of PanIN and PDAC. Moreover, since the growth term of the sensory axons only depends on the PanIN and PDAC, then the sensory axons tend to their threshold $\tau_{A_2}$. This is stated in the following proposition.
\begin{proposition}\label{prop:conv_a2}[Exponential convergence of the sensory axons]\\
Let $Q_0(0) >0$, $Q_1(0) = Q_2(0) = Q_3(0) =0$, $ A_1(0) \in (-\tau_{A_1}, \tau_{A_1})$ and $A_2(0)\in (0, \tau_{A_2})$. We assume \eqref{hyp:third} holds. Then $A_2$ tends exponentially fast to $\tau_{A_2}$ and for $t^*>0$ large enough, there exist a constant $C>0$ and a rate $r>0$ such that 
$$\forall t> t^*, \quad |A_2(t) - \tau_{A_2}| \leq C |A_2(0) - \tau_{A_2}| e^{-rt}.$$
\end{proposition}
\begin{proof}\\
Since the growth term of the sensory axons depends on the quantity of proliferating cells, the bounds on the cells populations $Q_2$ and $Q_3$ are one of the main information in order to deduce the exponential convergence of $A_2$. The following inequality gives us the bounds on the proliferating cells (PanIN and PDAC). Hence, it exists $t^* >0$ and two constants $0<c_y<C_y$ such that
$$\forall t> t^*, \quad c_y \leq Q_2(t) +Q_3(t) \leq C_y. $$
This results is detailed in Lemma \ref{lem:bnd_y} and its proof is postponed on the appendix for the sake of clarity. \\
We recall that 

\begin{align*}
\frac{d}{dt}| A_2(t)- \tau_{A_2}|  & = \text{sign}(A_2(t) - \tau_{A_2})(\bar{\alpha}_2 Q_2(t) +\bar{\alpha}_3 Q_3(t))A_2(t)\left(1 - \tfrac{A_2(t)}{\tau_{A_2}} \right) \\
&=- (\bar{\alpha}_2 Q_2(t) +\bar{\alpha}_3 Q_3(t))\frac{A_2(t)}{\tau_{A_2}}| A_2(t)-\tau_{A_2}| ,\\
 & \leq -\min(\bar{\alpha}_2, \bar{\alpha}_3)(Q_2(t) + Q_3(t))\frac{A_2(0)}{\tau_{A_2}}| A_2(t)-\tau_{A_2}|.
\end{align*}
The last inequality stands since for all time $t>0$ $A_2(t) \in (0, \tau_{A_2})$ (Proposition \ref{prop_pos}) and one can easily verify that for $A_2$ cannot decrease for $A_2(0) \in (0, \tau_{A_2})$.
We denote $y(t) = Q_1(t)+ Q_2(t)$ and we obtain by using its lower bound $c_y$ (Lemma \ref{lem:bnd_y}) that for $t^*>0$ large enough and for all $t > t^*$
\begin{align*}
\frac{d}{dt}| A_2(t)- \tau_{A_2}| & \leq -C(c_y, A_2(0))| A_2(t)- \tau_{A_2}|,
\end{align*}
where $C(c_y, A_2(0))>0$. Finally, we obtain the exponential convergence of $A_2$ toward $\tau_{A_2}$ using the Gronwall's Lemma.
\end{proof}
\subsubsection{Reduced system and its asymptotic behavior}
Under the conditions of Proposition \ref{prop:conv_a2}, it can be shown that the long-time dynamics of the system \eqref{eq:sys_abs} are governed and completely determined by the asymptotic of the following "reduced" system : 
\begin{equation}\label{eq:sys_abs_lim}
\left\lbrace
\begin{array}{c@{}l}
\frac{d}{dt} Q_2(t) = &\gamma_{2}Q_2(t)\left(1 -\tfrac{Q_2(t)+Q_3(t)}{\tau_{C}}+ \tfrac{A_1(t)}{\tau_{A_1}^C}+\tfrac{\tau_{A_2}}{\tau_{A_2}^C}\right) -f_2\left(A_1(t),\tau_{A_2}\right) Q_2(t)\\
\\
\frac{d}{dt} Q_3(t) = &\gamma_{3} Q_3(t)\left(1 -\tfrac{Q_2(t)+Q_3(t)}{\tau_{C}}+\tfrac{A_1(t)}{\tau_{A_1}^C}+\tfrac{\tau_{A_2}}{\tau_{A_2}^C}\right) + f_2\left(A_1(t),\tau_{A_2}\right) Q_2(t)\\
\\
\frac{d}{dt} A_1(t) = &\left(\alpha_2 Q_2(t)-\alpha_3 Q_3(t) \right)\left(1+\tfrac{A_1(t)}{\tau_{A_1}} \right)\left(1-\tfrac{A_1(t)}{\tau_{A_1}} \right)\\
\end{array}
\right.
\end{equation}

The system \eqref{eq:sys_abs_lim} constitutes the starting point in order to study the asymptotic behavior of the model \eqref{eq:Mod_Q0}-\eqref{eq:Mod_A2}. We first focus on the steady states (given by $\tfrac{dQ_2}{dt}=\tfrac{dQ_3}{dt}=\tfrac{dA_2}{dt}=0$) and their local stability.

\begin{proposition}\label{prop: ss}[Steady states of the system \eqref{eq:sys_abs_lim}]\\
Let \eqref{hyp:second} and \eqref{hyp:third} be true. We denote $C(x) = 1 + \frac{x}{\tau_{A_1}^C} +\frac{\tau_{A_2}}{\tau_{A_2}^C}$. Then the admissible steady states of the system are the following:
\begin{itemize}
\item $\qeq = 0$, $ \xeq = 0$, $\aeq = c$ where $c \in \big(-\tau_{A_1};\tau_{A_1} \big)$, which is linearly unstable,
\item $\qeq = 0$, $ \xeq = \tau_C \times C(\tau_{A_1})$, $\aeq = \tau_{A_1}$, which is linearly unstable,
\item $\qeq = 0$, $ \xeq = \tau_C \times C(-\tau_{A_1})$, $\aeq = -\tau_{A_1}$, which is linearly stable.
\end{itemize}
\end{proposition}

\begin{proof}\\
The condition $\tau_{A_1} <\tau_{A_1}^C$ implies that $C(A_1(t)) = 1 + \frac{A_1(t)}{\tau_{A_1}^C} + \frac{\tau_{A_2}}{\tau_{A_2}^C} >0$ for all time $t\geq 0$. Hence, the steady states must fulfill the following equations: 
$$ \gamma_2 \qeq \left( C(\aeq) - \frac{f_2^\infty}{\gamma_2} - \frac{\qeq +\xeq}{\tau_c} \right)= 0, $$
where $f_2^\infty = f_2(\aeq, \tau_2).$
There are two cases, either $\qeq = 0 $ or $\qeq = \tau_c \left( C(\aeq) - \frac{f_2^\infty}{\gamma_2} - \frac{\xeq}{\tau_C}\right) $. First, we study the case where $\qeq \neq 0$. It implies by the positiveness of the solution (Proposition \ref{prop_pos}) $$ C(\aeq) - \frac{f_2^\infty}{\gamma_2} - \frac{\xeq}{\tau_C} > 0$$ and then $$ C(\aeq) - \frac{f_2^\infty}{\gamma_2} > 0.$$
Hence, we have 
\begin{align*}
0 & = \gamma_3 \xeq \left( C(\aeq)  - \frac{\qeq +\xeq}{\tau_c} \right)  + f_2^\infty \qeq\\
 & = \frac{\gamma_3}{\gamma_2}f_2^\infty \xeq  + f_2^\infty \tau_C \left( C(\aeq) - \frac{f_2^\infty }{\gamma_2} - \frac{\xeq}{\tau_C}\right)
\end{align*}
and then $$\xeq = \frac{\gamma_2}{\gamma_2 - \gamma_3}\tau_C \left( C(\aeq )-\frac{f_2^\infty}{\gamma_2}\right). $$
Since $\gamma_2 < \gamma_3$, we have that $\xeq < 0$ which contradicts the positiveness property. Finally, we obtain that $\qeq =0 $ and the steady states of the Proposition \ref{prop: ss} are then determined by simple computations.\\
Hence, the linearization around the steady states $\qeq=0, \; \xeq$ and $\aeq$ give the following Jacobian matrix:
\begin{align*}
\begin{pmatrix}
\gamma_2(C(\aeq) -\xeq/\tau_C)-f_2^\infty & 0 & 0\\
-\gamma_3 \tfrac{\xeq}{\tau_C} + f_2^\infty & \gamma_3(C(\aeq) -2\xeq/\tau_C) & \gamma_3 \xeq / \tau_{A_1}^C\\
\alpha_2\left(1 - \left(\tfrac{\aeq}{\tau_{A_1}}\right)^2 \right) & -\alpha_3\left(1 - \left(\tfrac{\aeq}{\tau_{A_1}}\right)^2 \right) & 2 \alpha_3 \xeq \tfrac{\aeq}{(\tau_{A_1})^2}
\end{pmatrix}
.
\end{align*}
In the case $Q_2^\infty= \xeq = 0$ and $\aeq =c$, the eigenvalues of the Jacobian matrix are:
$$\lambda_1 = \gamma_2 C(\aeq) - f_2^\infty, \quad \lambda_2 = \gamma_3 C(\aeq), \quad \lambda_3 = 0 .$$
In the last case, since $\xeq = \tau_C C(\aeq)$ and $\aeq = -\tau_{A_1}$, the eigenvalues of the Jacobian matrix are:
$$\lambda_1 =-f_2^\infty , \quad \lambda_2=-\gamma_3C(\aeq)  ,\quad \lambda_3=2\alpha_3\tau_C C(\aeq)\frac{\aeq}{(\tau_{A_1})^2}  .$$ 
We conclude that the only linearly stable steady state is $\qeq= 0,\; \xeq=\tau_C C(-\tau_{A_1}),\; \aeq = -\tau_{A_1}.$
\end{proof}

Now, we study the global behavior of the limit system looking at the trajectories in the vector field. The system \eqref{eq:sys_abs_lim} is a modification of a competitive Lotka-Voltera system \cite{hofbauer1998evolutionary} for $Q_2$ and $Q_3$ coupled to a modified logistic equation for $A_1$. The following result is established.
\begin{theorem}[Long time behavior of the limit system \eqref{eq:sys_abs_lim}]
Let \eqref{hyp:second} and \eqref{hyp:third} be true. Let $Q_2(0) >0$, $Q_3(0) \geq 0$ and $A_1(0) \in \big(-\tau_{A_1} , \tau_{A_1} \big).$\
Then $$X(t) = \big( Q_2(t),\; Q_3(t),\; A_1(t)\big) \longrightarrow X^\infty =\big( 0,\; \tau_C C(-\tau_{A_1}), \; -\tau_{A_1} \big) \quad \text{for } t\rightarrow +\infty $$ 
where $C(-\tau_{A_1}) = 1 - \frac{\tau_{A_1}}{\tau_{A_1}^C} + \frac{\tau_{A_2}}{\tau_{A_2}^C}$.
\label{thm: ltb_lim}
\end{theorem}

\begin{proof}\\
First, we introduce the following functions:
\begin{itemize}
\item $ z(A_1, Q_2, Q_3) = C(A_1) - \frac{Q_2 +Q_3}{\tau_C} $, we also denote $z(t) = z(A_1(t), Q_2(t), Q_3(t))$ for the sake of simplicity,
\item $w(Q_2, Q_3) = \alpha_2 Q_2 - \alpha_3 Q_3$, we also denote $w(t) = w( Q_2(t), Q_3(t))$ for the sake of simplicity.

\end{itemize}

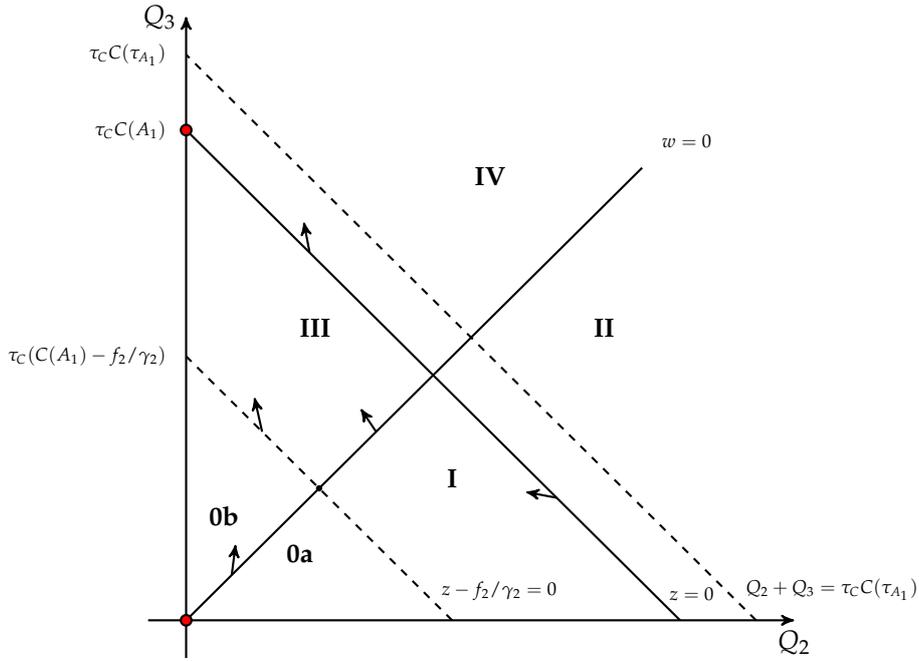
\begin{figure}[!ht]
\begin{center}
\begin{tikzpicture}[
     thick,
     >=stealth',
      empty dot/.style = { circle, draw, fill = white!0,
                           inner sep = 0pt, minimum size = 4pt },
     filled dot/.style = { empty dot, fill = red}
   ]
   \def\r{3}
   \draw[->] (-0.5,0) -- (8,0) coordinate[label = {below:$Q_2$}] (xmax);
   \draw[->] (0,-0.5) -- (0,8) coordinate[label = {left:$Q_3$}]  (ymax);
   \draw [] (0,0) -- (6,6);
   \draw [dashed] (3.5,0) -- (0,3.5);
   \draw [] (6.5,0) -- (0,6.5);
   \draw [dashed] (7.5,0) -- (0,7.5);
   \node ["above right:\scriptsize{$w=0$}"]  at (6,6) {};
   \node ["left:\scriptsize{$\tau_C(C(A_1) -f_2/\gamma_2)$}"] at (0,\r+0.5) {};
   \node ["above right:\scriptsize{$z-f_2/\gamma_2 = 0$}"] at (\r+0.1,0) {};
   \node [filled dot] at (0,6.5) {};
   \node [filled dot] at (0,0) {};
   \node ["left:\scriptsize{$\tau_C C(A_1)$}"] at (0,6.5) {};
   \node ["left:\scriptsize{$\tau_C C(\tau_{A_1})$}"] at (0,7.5) {};
   \node ["above right:\scriptsize{$z= 0$}"] at (6.1,0) {};
   \node ["above right:\scriptsize{$Q_2 + Q_3= \tau_C C(\tau_{A_1})$}"] at (7.1,0) {};
   \node ["\textbf{0a}"] at (1.5, 0.5) {};
   \node ["\textbf{I}"] at (3.5, 1.5) {};
   \node ["\textbf{II}"] at (5.5, 3.5) {};
   \node ["\textbf{IV}"] at (4, 5.5) {};
   \node ["\textbf{III}"] at (1.7, 3.5) {};
   \node ["\textbf{0b}"] at (0.5,1) {};
   \node at (1.75,1.75)[circle, fill=black, scale=0.25] {};
   \pgfmathsetmacro{\vx}{0.4}
   \pgfmathsetmacro{\vy}{0.4}
   \draw[->] (0.6,0.6) -- (0.25+\vy, 0.6+\vy);
   \draw[->] (2.5,2.5) -- (2.3, 2.8);
   \draw[->] (1.625,4.875) -- (1.55, 5.275);
   \draw[->] (4.875,1.625) -- (4.475, 1.7);
   \draw[->] (1,2.5) -- (0.9, 2.95);

\end{tikzpicture}
\caption{Scheme of the vector field projected on the $Q_2,Q_3$-plane.}
\label{fig: phase_por}
\end{center}
\end{figure}

Moreover, we introduce the following subdomains of $\R_+ \times \R_+ \times [-\tau_{A_1}, \tau_{A_1}]$ (see Figure \ref{fig: phase_por}) :
\begin{itemize}
\item the domain \textbf{0a} such that $ z-f_2/\gamma_2 \geq 0$, $w\geq 0$ and $Q_3 \geq 0$
\item the domain \textbf{I} such that $ z-f_2/\gamma_2 \leq 0$, $w\geq 0$, $Q_3 \geq 0$ and $z\geq 0$
\item the domain \textbf{II} such that $ z \leq 0$, $w\geq 0$ and $Q_3 \geq 0$
\item the domain \textbf{III} such that $ z-f_2/\gamma_2 \leq 0$, $w\leq 0$, $Q_2 \geq 0$ and $z \geq 0 $
\item the domain \textbf{IV} such that $ z \leq 0$, $w\leq 0$ and $Q_2 \geq 0$
\item the domain \textbf{0b} such that $ z-f_2/\gamma_2 \geq 0$, $w\leq 0$ and $Q_2 \geq 0$
\end{itemize}
In the Figure \ref{fig: phase_por}, the red dots are the admissible steady states. In order to prove the convergence of $X(t)$ toward $X^\infty$, we look at the trajectories in the different subdomains. We recall that it exists $m,\; M$ such that $0< m <f_2(t)< M$ for $t\geq 0$.\\
Now, let us suppose that it exists $ t_0>0$ such that $X(t_0) \in \textbf{0a}^\circ. $ Then  $\exists \delta >0$ such that for $t\in [t_0, t_0 + \delta)$,
$$\dt Q_2(t) >0,\quad \dt Q_3(t) >0, \quad \dt A_1(t) >0. $$
Also, 
\begin{align*}
\dt Q_3(t) &= \gamma_3 Q_3(t) z(t)  + f_2(t) Q_2(t),\\
& \geq m Q_2(t_0).
\end{align*}
Then, the trajectories cannot stay in $\textbf{0a}^\circ$, since $Q_3$ is bounded (Lemma \ref{lem:bnd_y}), and either the trajectories go to $\textbf{0b}$ or to \textbf{I}.
In addition, if the trajectories go from \textbf{0a} to \textbf{I}, it cannot go back again in \textbf{0a}. This result comes from the fact that $\dt Q_3 >0$ for $X\in \textbf{0a}\cup\textbf{I}$ and the fact that if the vector field points toward \textbf{I} on $\lbrace z - f_2/\gamma_2 =0 \rbrace$ when $Q_3 =\tau$ where $\tau>0$ then it points toward \textbf{I} on $\lbrace z - f_2/\gamma_2 =0 \rbrace\cap \lbrace Q_3 \geq \tau\rbrace$ (Lemma \ref{lem:vect_F}).

Let us suppose that it exists $t_1 >0 $ such that $X(t_1) \in \textbf{I}^\circ $. Then $\exists \delta >0$ such that for $t\in [t_1, t_1 + \delta)$,
$$\dt Q_2(t) <0,\quad \dt Q_3(t) >0, \quad \dt A_1(t) >0. $$ 
The trajectories cannot stay \textbf{I} since $Q_2, \; Q_3$ and $A_1$ are monotonous, bounded and there is no steady-states in the subdomain \textbf{I}. The vector field points inward on the surface $ \textbf{I} \cap\lbrace z=0\rbrace $ and it points outward on the surface $ \textbf{I} \cap\lbrace w=0\rbrace $. The trajectories can only go into the domain \textbf{III}.

Let us suppose $\exists t_2 >0$ such that $X(t_2) \in (\textbf{0b}\cup \textbf{III})^\circ. $ If $X(t_1) \in \textbf{0b}^\circ$ then  $\exists \delta >0$ such that for $t\in [t_2, t_2 + \delta)$,
$$\dt Q_2(t) >0,\quad \dt Q_3(t) >0, \quad \dt A_1(t) <0. $$
The trajectories cannot stay in \textbf{0b} and go to \textbf{III} since the null steady state is locally unstable (Proposition \ref{prop: ss}) and $Q_2$ and $Q_3$ are nondecreasing. Moreover, the vector field points toward \textbf{III} on $\lbrace  z - f_2/\gamma_2 =0 \rbrace$. \\
If $X(t_1) \in \textbf{III}^\circ$ then  $\exists \delta >0$ such that for $t\in [t_2, t_2 + \delta)$,
$$\dt Q_2(t) <0,\quad \dt Q_3(t) >0, \quad \dt A_1(t) <0. $$
The vector field  points outward on the surface $ \textbf{III} \cap\lbrace z=0\rbrace$ and inward on the surface $ \textbf{III} \cap\lbrace w=0\rbrace  $ (cf. Lemma \ref{lem:vect_F}). Either, the trajectories stay in \textbf{III} and converge toward $X^\infty$ or the trajectories go to \textbf{IV}.

If it exists $t_3>0$ such that $X(t_3)\in \textbf{IV}^\circ$, then $\forall t>t_3$ $\frac{d}{dt} Q_2(t) <0$ then $Q_2(t) < Q_2(t_3)$. Moreover, since $w(t)<0$ then $\frac{d}{dt}A_1(t) < 0 $ hence $A_1$ tends to $-\tau_{A_1}$. Also $z(t) <0$, it implies that 
\begin{align*}
\dt Q_2(t) &= \gamma_2 Q_2(t) \left(z(t) -\frac{f_2(t)}{\gamma_2} \right),\\
&\leq -\gamma_2 \frac{f_2(t)}{\gamma_2} Q_2(t),\\
&\leq -m Q_2(t), 
\end{align*}
and that $ Q_2(t) + Q_3(t) \geq \tau_C C(A_1(t)) \geq \tau_C C(-\tau_{A_1})$ hence $Q_2 + Q_3$ is decreasing and bounded below by $\tau_C C(-\tau_{A_1})$ and $Q_2$ tends to 0.
Finally, we conclude that once the trajectories enter \textbf{IV} they tend to $X^\infty$.
\medskip

In addition, if there exists $t_4>0$ such that $X(t_4) \in \textbf{II}^\circ$ then $\exists \delta >0$ such that for $t\in [t_4, t_4 + \delta)$, 
$$\dt Q_2(t) <0, \quad \dt Q_3(t) <0,\quad \dt A_1(t)>0.$$
Let us suppose that the trajectories remains in \textbf{II}. Since $Q_2$, $Q_3$ and $A_1$ are monotonous and bounded, $X$ converge to a point in \textbf{II}. However, it is absurd since there is no steady state in \textbf{II}. It implies that the trajectories leave the space \textbf{II} and enter either \textbf{I}, \textbf{III} or \textbf{IV}.

\end{proof}
\subsubsection{Asymptotic behavior of the complete system}
Once the convergence is established on the limit system, the global asymptotic behavior of the complete system is given by the following theorem.
\begin{theorem}[Long time behavior of the system \eqref{eq:sys_abs}]

Let $Q_0(0) >0$, $Q_i(0) \geq 0$ for $i =1,\;2,\;3$, $ A_1(0) \in (-\tau_{A_1}, \tau_{A_1})$ and $A_2(0)\in (0, \tau_{A_2})$. Let \eqref{hyp:second} and \eqref{hyp:third} be true.
Then $$X(t) = \big(Q_0(t),\; Q_1(t),\; Q_2(t),\; Q_3(t),\; A_1(t),\; A_2(t)\big) \longrightarrow \big(0,\; 0,\;  0,\; \tau_C C(-\tau_{A_1}), \; -\tau_{A_1},\; \tau_{A_2} \big) \quad \text{for } t\rightarrow +\infty $$ 
where $C(-\tau_{A_1}) = 1 - \frac{\tau_{A_1}}{\tau_{A_1}^C} + \frac{\tau_{A_2}}{\tau_{A_2}^C}$.

\label{thm: ltb}
\end{theorem}
\begin{proof}\\
We introduce the function $F:\; (0,\infty)\times\mathbb{R}^3 \mapsto \mathbb{R}$ such that for $Y =  (Q_2, Q_3, A_1)$:
\begin{align*}
F(t,Y)& = \begin{pmatrix}
\gamma_{2}Q_2\left(1 -\tfrac{Q_2+Q_3}{\tau_{C}}+ \tfrac{A_1}{\tau_{A_1}^C}+\tfrac{A_2(t)}{\tau_{A_2}^C}\right) +f_1\left(A_1\right) Q_1(t)-f_2\left(A_1,A_2(t)\right) Q_2\\
\gamma_{3} Q_3\left(1 -\tfrac{Q_2+Q_3}{\tau_{C}}+\tfrac{A_1}{\tau_{A_1}^C}+\tfrac{A_2(t)}{\tau_{A_2}^C}\right) + f_2\left(A_1,A_2(t)\right) Q_2\\
\left(\alpha_{1}Q_1(t)+\alpha_2 Q_2-\alpha_3 Q_3 \right)\left(1+\tfrac{A_1}{\tau_{A_1}} \right)\left(1-\tfrac{A_1}{\tau_{A_1}} \right)
\end{pmatrix}.
\end{align*}
Moreover, Proposition \ref{prop_pos} gives the exponential convergence of $Q_0$ and $Q_1$ toward 0 and Proposition \ref{prop:conv_a2} gives the exponential convergence of $A_2$ toward $\tau_{A_2}$. Hence, we obtain that 
$$F(t,Y) \rightarrow G(Y) \quad \text{as }t \rightarrow \infty \quad \text{uniformly locally in } Y \in \mathbb{R}^3, $$
where 
\begin{align*}
G(Y)& = \begin{pmatrix}
\gamma_{2}Q_2\left(1 -\tfrac{Q_2+Q_3}{\tau_{C}}+ \tfrac{A_1}{\tau_{A_1}^C}+\tfrac{\tau_{A_2}}{\tau_{A_2}^C}\right) -f_2\left(A_1,\tau_{A_2}\right) Q_2\\
\gamma_{3} Q_3\left(1 -\tfrac{Q_2+Q_3}{\tau_{C}}+\tfrac{A_1}{\tau_{A_1}^C}+\tfrac{\tau_{A_2}}{\tau_{A_2}^C}\right) + f_2\left(A_1,\tau_{A_2}\right) Q_2\\
\left(\alpha_2 Q_2-\alpha_3 Q_3 \right)\left(1+\tfrac{A_1}{\tau_{A_1}} \right)\left(1-\tfrac{A_1}{\tau_{A_1}} \right)
\end{pmatrix}.
\end{align*}
Using the results of Theorem \ref{thm: ltb_lim}, we have that the $\omega$-limit set (see Appendix \ref{app:asymp_aut_sys}, equation \eqref{eq:omeg_set}) of the limit system \eqref{eq:sys_abs_lim} is restricted to
$$\omega(0, Y_0) = \lbrace(0, \tau_C C(-\tau_{A_1}), - \tau_{A_1}) \rbrace. $$
Then, using Theorem \ref{thm:asymp_auto_2} on the asymptotically autonomic differential equations (see Appendix \ref{app:asymp_aut_sys}),we deduce the following results on the solutions of system \ref{eq:sys_abs}  
$$\big(Q_0(t),\; Q_1(t),\; Q_2(t),\; Q_3(t),\; A_1(t),\; A_2(t)\big) \longrightarrow \big(0,\; 0,\;  0,\; \tau_C C(-\tau_{A_1}), \; -\tau_{A_1},\; \tau_{A_2} \big) \quad \text{for } t\rightarrow +\infty .$$
\end{proof}
\subsection{Interpretation of the mathematical results.}

The asymptotic behaviour of the system corresponds to the pathological case: the tumor cells are proliferating and it remains only sensory axons in the domain. This steady-state is corroborated in the observation. Furthermore, the maximum amount of PDAC cells is given by $$\tau_C C(-\tau_{A_1}) =\tau_C \left(1 - \frac{\tau_{A_1}}{\tau_{A_1}^C} + \frac{\tau_{A_2}}{\tau_{A_2}^C} \right).$$
This quantity corresponds to the maximum carrying capacity of the domain and depends positively on the amount of sensory axons and negatively on the amout of autonomic axons. This result resonates with the fact that the sensory axons have a pro-tumoral effect and a relatively high density in PDAC tumors whereas the autonomic axons have an anti-tumoral effects  and a relatively low density in the tumoral tissues during the late stages of the cancer development.

\section{Calibration of the system and results}\label{sec:calib}
Next, we calibrated the model with biological information captured in the data. The calibration seems to be in general a subjective approach. In our case, the large number of input parameters renders it challenging and highly time-consuming.
We use an optimization method in order to reduce the user intervention and obtain a more objective selection of parameters.\\ 
We first introduce and study the biological information aggregated in the data. Then, we construct an objective function that integrates the biological assumptions and the experimental data. We study the identifiability of the parameters in regard to this objective function. Finally, we obtain sets of parameters which calibrate reasonably well the model and we discuss the sensitivity of these parameters. 

\subsection{Biological assumptions and observation}
 The observation aggregates experimental data and empirical knowledge such as experimentally validated hypotheses (time of appearances, etc.). We gather all the biological information available on the process and formalize into the vector of observation denoted  $y^{*}$ and the vector of chronological parameters $(t_i)_{i=1 \hdots 5}$.

\subsubsection{Experimental data 1}

The observation used as referential in order to calibrate the model is represented in Figure \ref{fig:Prop_Mice1}. 
\begin{figure}[ht!]
\begin{center}
\includegraphics[width=\textwidth]{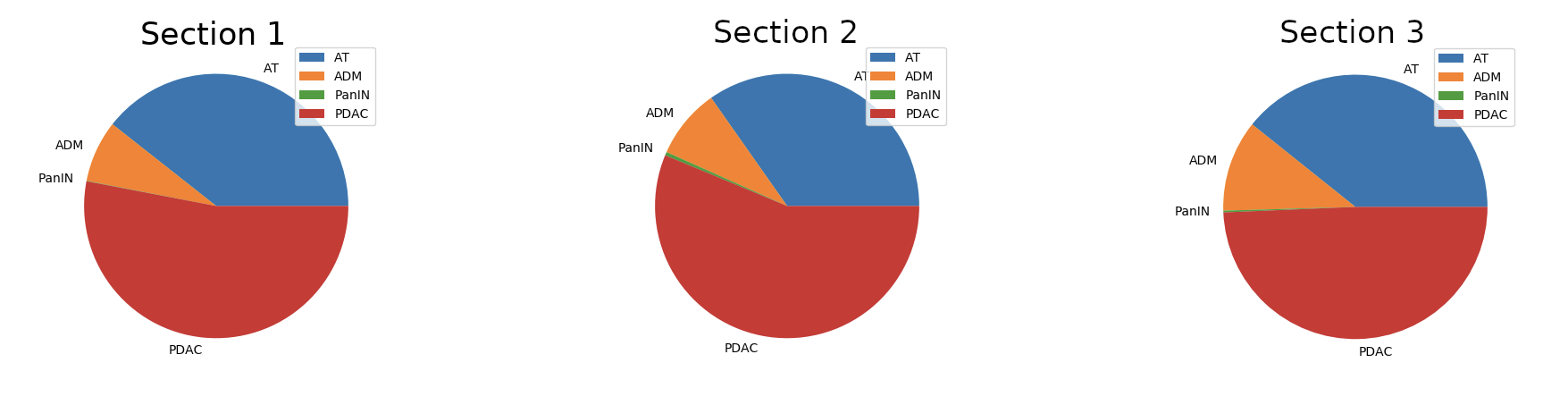}
\end{center}
\caption{{\textbf{Proportion of cell populations in the pancreas of a KIC mouse at 45 days.} The proportions of acini, ADM, PanIN and PDAC lesions was calculated on 3 sections. Each column represents the data assimilated to one section.}}
\label{fig:Prop_Mice1}
\end{figure}
Data were obtained from the KIC (\textit{LSL-KrasG12D/+; Cdkn2a (Ink4a/Arf)lox/lox; Pdx1-Cre}) transgenic mouse model of PDAC and describe the percentage of acinar tissue, ADM, PanIN and PDAC in histological sections throuth the pancreas of {a} 6.5 weeks-old mouse (45 days). Quantification data have been published and are available in the source data file of \cite{guillot2020sympathetic}.
{Because of the variability between samples, we choose to aggregate all the information contained in each section of the same mouse. We consider the observation as the proportion of each compartment averaged over all samples of the same mouse and denote these experimental data} $y^*_1$ for the Acini, $y^*_2$ for the ADM, $y^*_3$ for the PanIN and $y^*_4$ for the PDAC. 
In addition, the time of tissue harvest is 45 days and is denoted $t_f$ in the following.

\begin{remark}
One can reasonably object that considering a larger number of tissue sections is consistent with the fact that the domain in the mathematical model is the whole pancreas and its immediate environment. However, one of the purpose of this work is to be able to use human biopsies as data. In the case of human data, the number of biopsies is limited and it becomes interesting to develop generic models and methods that make up for the lack of information. 
\end{remark}

\subsubsection{Experimental data 2}
Data on axonal density were obtained by quantifying innervation in the pancreas of control and KIC mice. Whole pancreases were immunostained with antibodies specific to each PNS neuron subtype. Sympathetic axons were immunolabeled with an antibody against tyrosine hydroxylase (TH), parasympathetic axons with an antibody against vesicular acetylcholine transporter (VAChT) and sensory fibers with an antibody against calcitonin gene-related peptide (CGRP). The tissues were then imaged using LSFM to allow 3D visualisation of the neuronal networks. {LSFM} images were processed with Imaris software. Regions of Interest (ROIs: Acini, ADM, PanIN or PDAC) were segmented based on the autofluorescence signal of the tissue and their volumes were measured. Axonal networks were manually reconstructed using the Imaris “Filament tracer” tool. The “dendrite length sum” was collected for each ROI. The axon density was calculated as follows: dendrite length sum ($nm$) / volume of ROI ($\mu m^3$). {The full protocol and the quantification data for the sympathetic axons are available in the source data file of \cite{guillot2020sympathetic}.The quantification data for parasympathetic and sensory fibers are performed using the same protocol. Densities of autonomic axons are the sum of sympathetic and parasympathetic axons in each ROI. Axonal data are measured at day 45. We consider the observation as density in the total tumor volume and denote these experimental data $y^*_5$ for the autonomic axons, $y^*_6$ for the sensory axons and $A_1^{eq}$ the density of autonomic axons of the control mouse. The experimental data expressed in $nm(\mu m)^{-3}$ are as follows :
$$y^*_5 = 0.1077  , \quad y^*_6 = 0.1468 ,\quad A_1^{eq} = 0.0099. $$}

\subsubsection{Biological assumptions on the chronological process}
\textbf{Initial conditions :} there are only healthy cells (Acini) at initial time and there is a negligible amount of sensory axons in the pancreas and a small amount of autonomic axons at initial time. These conditions are formalized by \eqref{eq:init_cond}.{ In addition, the neuroplastic changes and the cancer progression are relatively negligeable at early stages (before two weeks of age in the KIC model).} We denote $t_0$ the parameter (in days) corresponding to the initial time for the model simulations. Without loss of generality, we consider $t_0 =10$ days. \\

\textbf{Chronological appearances :} Based on the previous characterisation of the KIC model (\cite{aguirre2003activated} and our personal observations), the first appearance of ADM is around time $t_1 = 17$ days, of PanIN around time $t_2 =21 $ days and PDAC around time $t_3 = 35$ days. The density of autonomic axons increases in ADM regions, peaks in PanIN and decreases in PDAC, while sensory axons are rarely observed in PanIN, but have a high density in PDAC. We therefore empirically set the time of appearance of $A_1$ and $A_2$ at $t_4 = 18$ days and $t_5 = 30$ days, respectively.  
\\

\subsection{Parameters calibration with an optimization procedure}
The dynamics of the system depends very strongly on the choice of parameters. This choice is based on the calibration of the model. In other words, it is first necessary to quantify the distance between the outputs of the model and the biological data and then to find the parameters that minimize this distance in an objective way. That is the reason why, testing the goodness of fit through an optimization procedure impose itself as a rigorous method. The following section describes a data-driven process which minimizes the deviation between the observation and the model and which gives an objective calibration of the parameters.

\subsubsection{Objective function}
The measurement of this deviation is made possible by the objective function or also called the cost function which is denoted as $\mathcal{G}$. The inputs of $\mathcal{G}$ are the parameters of \eqref{eq:Mod_Q0}-\eqref{eq:Mod_A2} (see Table \ref{tab:parameters}) denoted as the vector $\theta$. This cost function integrates the biological assumptions and the comparison to the experimental data. The further away from the experimental data the trajectories are, the higher the cost function is. Hence, finding the parameters which minimize the cost function is equivalent to the calibration of the model. We recall that $y^*$ are the observation used to calibrate the model. Hence, $y^* \in \mathcal{K}$ where $\mathcal{K}$ is a compact set in $\mathbb{R}^6$. The vectors of parameters $\theta^*$ which minimize the function $\mathcal{G}$ give the optimal calibration of the system where the biological constraints and hypothesis are taken into accounts. The cost function is described by the following equation :
\begin{align}
\begin{split}
\mathcal{G}(\theta)=&\underbrace{\sum\limits_{k=0}^{3} \left(a_k \int_{t_f-3}^{t_f+3} \left(\frac{Q_k(s|\theta)}{(Q_0+Q_1+Q_2+Q_3)(s|\theta)} - y_{k+1}^{*}\right)^2ds \right)}_\text{\parbox{6.5cm}{Measure of the deviation between the numerical simulation and the experimental data 1}}\\
& + \underbrace{a_4 \int_{t_f-3}^{t_f+3} \left(A_1(s|\theta) + A_1^{eq}- y_5^{*}\right)^2ds  + a_5 \int_{t_f-3}^{t_f+3} \left(A_2(s|\theta) - y_6^{*}\right)^2ds}_\text{\parbox{7.5cm}{Measure of the deviation between the numerical simulation and the experimental data 2}} \\
& + \underbrace{\sum\limits_{k=1}^{3}\left(b_k  \int_{t_0}^{t_k}(Q_k(s|\theta))^2 ds\right) 
 + b_4\int_{t_0}^{t_4}(A_1(s|\theta))^2 ds + b_5\int_{t_0}^{t_5}(A_2(s|\theta))^2 ds}_\text{\parbox{7.5cm}{Penalization term ensuring the chronological assumptions}}.
\end{split}
\label{eq:cost_func}
\end{align}
 The parameter $t_f$ corresponds to the time in days of data extraction from in vivo experimentations (see Figure \ref{fig:Prop_Mice1}), hence $t_f =45$. The definition of \eqref{eq:cost_func} also includes the interval of days $[t_f -3; t_f +3]$ and the parameters $(a_k)_{k =0, \hdots, 5}$, $(b_k)_{k=1,\hdots, 5}$. 
The time interval $[t_f -3; t_f +3]$ allows us to use the $L_2$-norm squared to compare the simulated trajectories to the observation. However, it implies that the experimental data are supposed to be true on a six days interval around $t_f$. One can reasonably justify this assumptions by the fact that this non-local norm regularizes the observation from its inherent biological chaos. Chaos means that biological phenomenon, processes or experiments are extremely sensitive to small perturbations. In our case, it is translated by the variability of chronological appearances of phenomenon between mice. The $L^2$-norm smooths the variability in time in the objective function and is less sensitive to the observation variability.\\ 

\begin{remark} We introduce $a_k$ and $b_k$ as normalization parameters in $\mathcal{G}$. The objective function $\mathcal{G}$ is a sum of positive terms where each of these terms have a contribution to the final cost. These contributions are individually linked to a specific part of the information on the biological process. However, there are discrepancies between the different parts, for instance, $y_2^* \ll y_3^*$ (cf. Figure \ref{fig:Prop_Mice1}) or $t_1 \ll t_3$. Some data can be falsely assimilated as outliers. Hence, we consider the following:
$$a_k = \frac{1}{6 y_k^*}, \quad b_k = f_k(y_k^*)\frac{1}{t_k - t_0},$$
where $f_k$ adjusts the coefficient $b_k$ in regard to $y_k^* $. Ultimately, the normalization coefficients make all contributions to the cost function equal.
\end{remark}

\subsection{Identifiability and sensitivity of the parameters}
We study the identifiability of the model to understand the degree to which the parameters can be constrained to a unique value or a reasonable range of values given the data available. Since multiple model parametrizations produce similar behaviour (see section \ref{sec:opti}), we use profile likelihood methods (cf. \cite{kreutz2013profile,roosa2019assessing}) to refine the interpretation of the estimators (see section \ref{sec:bio_inter}). 

\subsubsection{Optimization challenges}\label{sec:opti}
\begin{figure}[ht!]
   \centering
   \includegraphics[width=\textwidth]{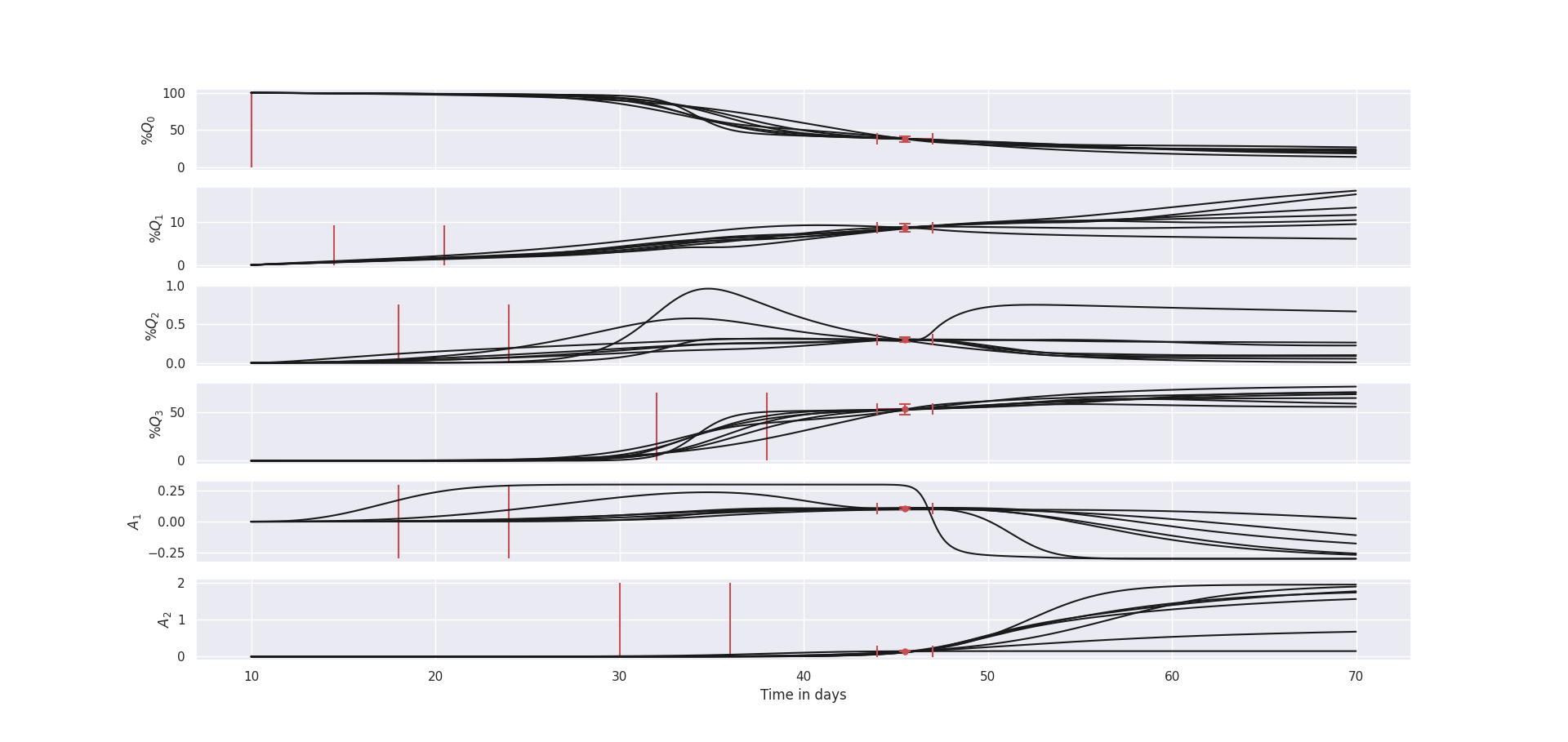}      
\caption{\textbf{Evolution of cell populations and axons in the pancreas for 7 different sets of parameters.} {\small The x-axis is the time in days. For the top four graphs, the y-axis corresponds to percentage of a specific population in the pancreas. The vertical red lines correspond to the time of appearance within a tolerance of six days. The red dots at day 45 correspond to the observation $y^*$ for the mouse 1 (see Figure \ref{fig:Prop_Mice1}) and and the confidence intervals of the data are represented by the red segments. Each curve on the same graph is associated to its set of parameters (see Table \ref{tab:par_Opti_1}).  All sets of parameters have almost the same optimization score (under 10), however the dynamics are different. } }
\label{fig:Opti_1}
\end{figure}

The optimization problem leads to mathematical and numerical challenges. With such experimental data and such a large set of parameters, we can not except to calibrate the model uniquely. It leads to multiple local minimizers and convergence issues. We use the Covariance Matrix Adaptation Evolution Strategy (CMA-ES see \cite{hansen2006cma,jastrebski2006improving}) which is based on a derivative-free evolution algorithm in order to overcome the obstacles linked to the size of the optimization problem and to the non-trivial dependence of the cost function on the parameters. The results of the optimization procedure give a set of acceptable parameters. The evolution of cell populations can be simulated for each of these parameters' vector (cf. Figure \ref{fig:Opti_1}). It is interesting to note that the trajectories in the numerical simulation show different behaviours for the same cost (see Figure \ref{fig:Opti_1}).\\
Recall that the expected asymptotic behaviour corresponds to the depletion of all healthy and precancerous cells and the saturation of PDAC. Furthermore, the autonomic axons density is supposed to vanish whereas the sensory axons density converges to a threshold (see Section \ref{sec:model}). The numerical simulations confirm that the time near the steady state and the time of the experimental data 1 are not on the same time frame. Moreover, some trajectories tend faster to the equilibrium with the same cost function values than others. For instance, we observe in Figure \ref{fig:Opti_1} that some sets of parameters correspond to a faster convergence rate: 
\begin{itemize}
\item We see that one trajectory of $Q_1$ is decreasing after time $t_f$.
\item Two trajectories of $Q_2$ are increasing then decreasing with a peak values between the range of the appearance time of the PDAC.
\item We distinguish that two trajectories of $A_1$ are increasing, reach their maximum value around 35 days and then  are decreasing relatively fast after 45 days.
\end{itemize}
Further analysis are needed in order to classify which parameters tuples are relevant.\\

\subsubsection{Discrete formulation of the objective function and link with the likelihood function}
Consider the continuous solution $X(t)$ of the system formed by the equations \eqref{eq:Mod_Q0}-\eqref{eq:Mod_A2} $$X(t) = (Q_0(t),Q_1(t),Q_2(t),Q_3(t),A_1(t),A_2(t)) = (X_k(t))_{k=1 \hdots 6}.$$ 
Recall that $X(t)$ is included in a compact subset of $\mathbb{R}^6$ denoted $\mathcal{K}$ (see Section \ref{sec:model}). Moreover, the time interval of interest is finite and we denote $T$ its upper bound and $I = [0,T)$. We denote $(s_i)_{0, \hdots,N-1} $ the finite sequence of length $N$ used to to discretize the time interval $[0, T)$. In order to calibrate the  parameters of the system \eqref{eq:Mod_Q0}-\eqref{eq:Mod_A2}, we introduce the \textit{model function} $g: \; \mathcal{K} \mapsto \; \mathbb{R}_{+}^6$. It allows us to confront the simulations to the measurement values and to the biological assumptions. In the first hand, we denote $g_k$ as the $k$-th coordinate of the model function denoted:
\begin{equation*}
g_k \left(s_i|\theta  \right) = 
	\begin{cases}
		\begin{cases}
		\frac{X_k}{\sum_j Q_j} (s_i) & \text{ if } k =1, \hdots,4\\
		X_4(s_i) + A_1^{eq} & \text{ if } k=5\\
		X_5(s_i) & \text{ if } k=6\\
		\end{cases} & \text{ if } s_i \in [t_f-3 ; t_f +3] \\		
		X_k(s_i) & \text{ if } s_i \in[0;t_k]\\
	\end{cases}
\end{equation*} 
where $A_1^{eq}$ is the initial equilibrium state of autonomic axons. In the second hand, we denote $\tilde{y}_k$ the $k$-th coordinate of the statistical observation (i.e. the quantitative value gathering the measurement values and the biological assumptions):
\begin{equation}
\tilde{y}_{k,i} = 
	\begin{cases}
	y^*_k & \text{ if } s_i \in [t_f-3;t_f +3]\\
	0 & \text{ if } s_i \in [0,t_k]
	\end{cases}
\label{eq:model_to_data}
\end{equation}
where $i = 0, \hdots, N-1$, $y^*$ is the vector of the experimental measurements and $t_k$ is the time of appearance of the $k$-th variable of X.
Finally, we denote the observables as following :
\begin{equation}
y_i(\theta) =g\left(s_i| \theta\right)+ \epsilon_i \quad \text{for } i=0,\hdots,N-1,
\label{eq:observable}
\end{equation}
where $\theta$ is the vector of parameters of the model and $\epsilon_i$ is the independent noise.\\

Finally, we denote $\Chi_2$ the objective function which measures the agreement of experimental data with the observables predicted by the model :

\begin{equation}
\Chi_2(\theta) = \sum\limits_{i=0}^{N-1}\frac{1}{c_i}  \| y_{i}(\theta)-\tilde{y}_{i} \|_2^2,
\label{eq:cost_func_disc}
\end{equation}
where $\tilde{y}_i$ is the vector composed by the $\tilde{y}_{k,i}$ (see \eqref{eq:model_to_data}) and $c_i$ are coefficients gathering the corresponding measurements errors and normalization terms coming from the discrepancy of the data. The objective function described in \eqref{eq:cost_func_disc} corresponds to the objective function $\mathcal{G}$ given in \eqref{eq:cost_func} in a discrete time frame and considering the experimental noise. This function is also assimilated to $-2LL$, where $LL$ is the \textit{log-likelihood}. Hence, the minimization problem is equivalent to the \textit{maximum likelihood estimation} problem. It is a widely documented problem in the literature and has beneficial properties like efficiency and consistency \cite{van2000asymptotic}. 

\subsubsection{Profile likelihood to identify the parameters}

The issue remains the large number of parameters and the "small" amount of data at our disposal. Therefore, the study of the identifiability of the parameters gives further knowledge about the fit between the model and the observation and helps to interpret the outputs given by the model. 
Since the optimization problem is linked to the maximization of the likelihood, the \textit{profile likelihood} solves the identifiability problem \cite{murphy2000profile}. This one-dimensional projection is performed to visually evaluate whether different values of the same parameter give similar outputs. We recall that the vector $\theta$ denotes the parameters of the model .
Hence, using the objective function \eqref{eq:cost_func_disc}, the impact of the value of the specific parameter $\theta_j$ for fitting the model to the observation is assessed by the following profile likelihood :
\begin{equation}
\mathcal{P}_j(p) = \min\limits_{\theta \in \lbrace \theta\;  |\; \theta_j =p \rbrace} \Chi_2(\theta),
\label{eq:prof_lik}
\end{equation}
where the objective function \eqref{eq:cost_func_disc} is evaluated as a function of the values $p^j$ taken by the parameter $\theta_j$ while all others parameters $ \theta_i,\; i\neq j$ are reoptimized. This one-dimensional representation of the likelihood \eqref{eq:prof_lik} can be geometrically interpreted in order to assess the identifiability of the parameter $\theta_j$ \cite{kreutz2013profile, raue2009structural}. 
For instance, a flat profile for $\mathcal{P}_j$ corresponds to a \textit{structural non-identifiability} for the parameter $\theta_j$. It implies that the parameter is non-unique for the minimization of the objective function. Eliminating the non-uniqueness requires more data or additional data on different quantities. If $\mathcal{P}_j$ has a minimum but is flat on one side then the parameter $\theta_j$ is considered as a \textit{practically non-identifiable} parameter. 
It implies that the data do not contain significant enough information about the parameter. The parameter value cannot be restricted to a precise value. Similarly, new experiments leading to additional data are required in order to rigorously estimate the parameters.
However, if $\mathcal{P}_j$ describes a curve with unique minimum in a realistic range of values then the parameters $\theta_j$ is \textit{identifiable}. Also, if some knowledge is assumed on the experimental noise $\epsilon_j$, a finite confidence interval can be computed and gives an asymptotic validation of the identifiability of the parameter $\theta_j$ (see \cite{murphy2000profile,kreutz2013profile,raue2009structural}).\\
An implementation of the profile likelihood has been performed where $\theta_j$ takes $20$ distinct values denoted $\left\lbrace p_i^j | i=1 \hdots 20\right\rbrace$ which cover its range (see Table \ref{tab:parameters}). For each $p_i^j$, the minimization problem \eqref{eq:prof_lik} has been solved numerically with 50 different initial conditions on the parameters $ \theta_i,\; i\neq j$ in order to fully explore the optimization domain (see Appendix \ref{app:ident_prof_lik}). 
This Monte-Carlo approach allows to obtain several sets of admissible parameters for each $p_i^j$. It gives a qualitative criterium to validate the identifiability of a parameter : if the median of the costs $\mathcal{P}_j$ admits a distinct minimum at $p^{j,*}$ then we consider the parameter to be identifiable and its value to be $p^{j,*}$.
 {It is reasonable to consider as estimator the identifiable parameters $p^{j,*}$, however these values are highly dependent on the choice of the discretization of the ${\left(p^j_i\right)_i}$. This uncertainty related to this choice is significantly reduced when we focus on the \textit{confidence range} obtained through the procedure, i.e. the range where the identifiable parameters give reasonable results concerning the deviation between the observations and the model.}

\begin{figure}[ht!]
\begin{center}
\includegraphics[width=\textwidth]{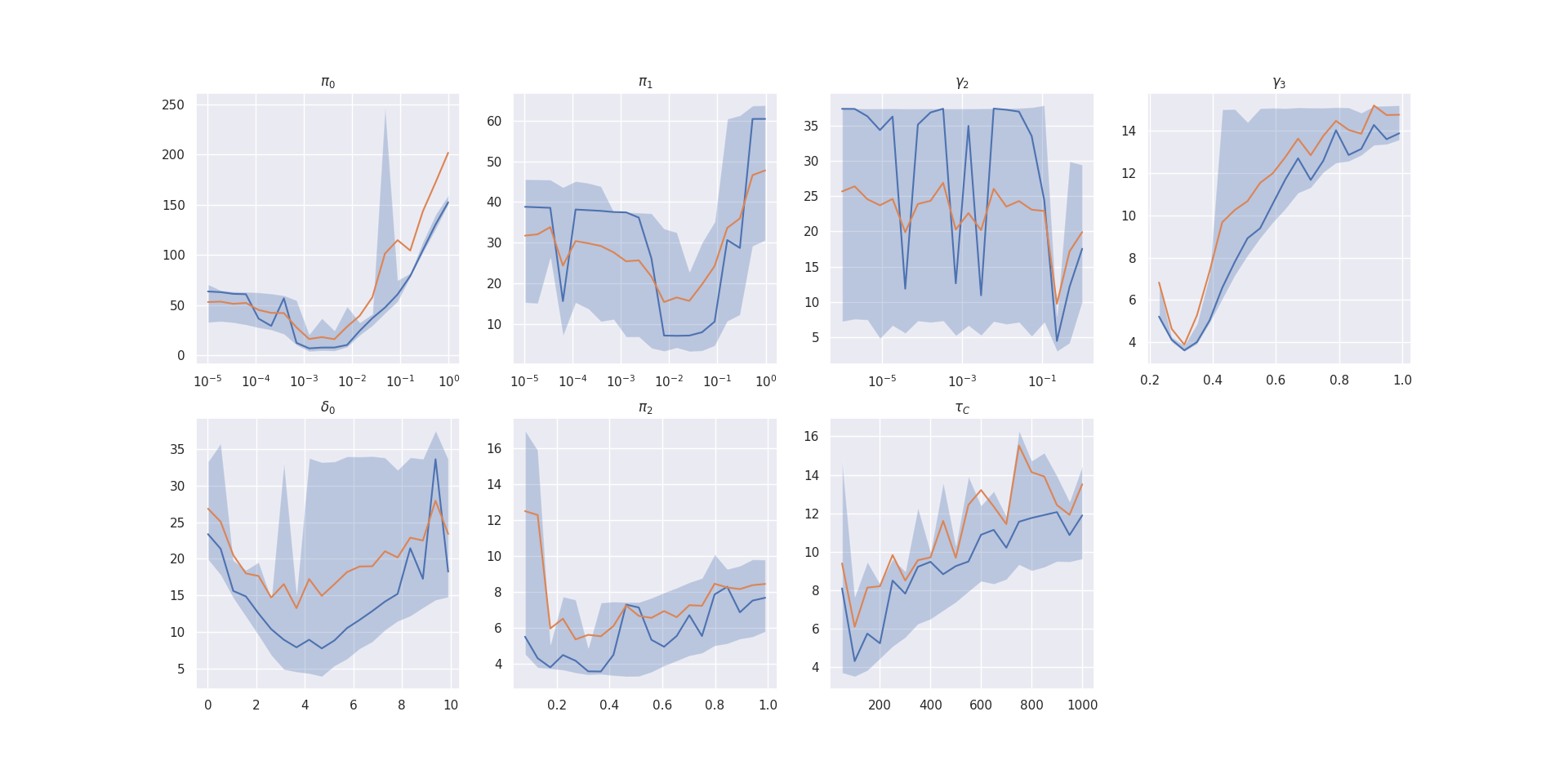}
\end{center}
\caption{\textbf{Numerical results of the identifiability problem using the profile likelihood.} {\small Each plot corresponds to a specific identifiable parameter ($\pi_0$, $\pi_1$, $\pi_2$, $\gamma_2$, $\gamma_3$, $\tau_C$, $\delta_0$).  The x-axis describes the values taken by the fixed parameter $\theta_j$. The y-axis denotes the corresponding fitting error defined in \eqref{eq:prof_lik}. The gray area represents 75$\%$ of the results between the first and last quartile (see Appendix \ref{app:ident_prof_lik}). The median (in blue) and the mean (in orange) for 50 iterations of the minimization problem show that each parameters admits a distinct minimum value under 10 for the fitting error.}}
\label{fig:ident_graph}
\end{figure}

As a result, we clearly distinguished seven identifiable parameters in our model (see Figure \ref{fig:ident_graph}):  the progression transfer rates $ \pi_0, \; \pi_1$ and $\pi_2$, the proliferation rates of the PanIN and the PDAC respectively $\gamma_2$ and $\gamma_3$, the saturation term for the proliferation of the PanIN and PDAC $\tau_C$ and finally the amplitude of the Michaelis-Menten term describing the effect of the PanIN and the PDAC on the progression of the Acini into ADM $\delta_0$. These parameters and their estimated values are gathered in the Table \ref{tab:ident_par}.

\begin{table}[ht!]
\begin{center}
\begin{tabular}{|l|c|c|c|c|c|c|c|}
 \hline
  & $\pi_0$ & $\pi_1$ & $\pi_2$ & $\gamma_2$ & $\gamma_3$ & $\tau_C$ & $\delta_0$\\
  \hline
  Estimated value & 0.0025 & 0.00942 & 0.176 & 0.229 & 0.31 & 100 & 3.65\\
  \hline
  Confidence range & $(10^{-4},10^{-2})$ & $(10^{-3},10^{-1})$ & $(10^{-1},0.5)$ & $(10^{-1},0.3)$ & $(0.2,0.4)$ & $(50,200)$ & $(3,5)$\\
  \hline
\end{tabular}
\end{center}
\caption{\textbf{Values of the identifiable parameters.} }
\label{tab:ident_par}
\end{table}
The identifiability of these seven parameters is further verified by comparing the cost distribution of the objective function in two cases: when we fix the identifiable parameters and minimize on all the remaining parameters and when all the parameters are free during the optimization. In the first case, the cost distribution is located around 4 (close to the minimal value obtained in the numerical simulation). In the second case, the cost distribution is much more spread (see Figure \ref{fig:hist_ident} in the Appendix). It confirms that the parameters are correctly estimated thanks to the data available since the distribution of fitting errors is concentrated on small values. Further details about the numerical methods are postponed in the Appendix \ref{app:ident_prof_lik}. In the end, this method allows to obtain a data-driven identifiability criterion without making assumptions on the experimental noise. It also gives estimators and their confidence range (see Table \ref{tab:ident_par}) for the identifiable parameters of the model.

\subsubsection{Biological interpretations} \label{sec:bio_inter}
{\textbf{Interpretation of the identifiability results.} This numerical study around the identifiability of the parameters of the model gives interesting insights about the experimental data and the biological processes behind it. First, we observed that the transfer rates from Acini to ADM $\pi_0$, ADM to PanIN $\pi_1 $ and PanIN to PDAC $\pi_2 $ are identifiable. Moreover, the estimates for the transfer rates respect the following order relation $$\pi_2 > \pi_i, \quad i=0,1. $$
It implies that the progression process appears to go faster at the late stage of the PDAC development.\\
We also obtain estimates for the proliferation rates of the PanIN $ \gamma_2$ and the PDAC $\gamma_3$. It implies that the data combining the proportion of cell populations at 45 days and the knowledge of the first time of appearance of each cell population give a sufficient amount of knowledge in order to calibrate the speed of the proliferation mechanism in the model. 
However, the threshold parameters $\tau_c$ is likely to be underestimated in regards to the data at our disposal. The value of this threshold is of a different order of magnitude compared to the other parameters. Therefore, its range of acceptable estimates is also wider compared to the other parameters and additional data are needed in order to compute a finer estimate.\\

\textbf{Sentivity analysis results.}The other parameters seem unidentifiable but a variance-based sensitivity analysis (see \cite{sobol1990sensitivity,sobol2001global}) allows us to acquire additional information. In order to perform this analysis, we restrict ourselves to the parameters linked to the effect of the axons on the cell population concentrations:  the parameters in the transition rates $\beta_1, \; \beta_2$ and $\delta_2$, the parameters in the proliferation terms $\tau_{A_1}^C$ and $\tau_{A_2}^C$ (c.f. Table \ref{tab:parameters}). 
The other parameters are fixed and are given by the last set of parameters in the Table \ref{tab:par_Opti_1}. The choice of this specific set of parameters is based on qualitative considerations such as its associated cost given by \eqref{eq:cost_func}, the fact that the associated trajectories are similar to the expected behavior (e.g. the decay of $A_1$, etc).
Hence, the inputs are these five parameters, which are denoted $$\vartheta = \left(\beta_1, \beta_2, \delta_2, \tau_{A_1}^C, \tau_{A_2}^C\right).$$ 
Concerning the output of the sensitivity analysis, our focus is on the PDAC cells. We introduce the following indicator of the variation of the PDAC cells:
\begin{equation}
    \mathcal{V} = 1 -\frac{\int_{t_0}^{t_F}Q_3(s;\vartheta)ds}{\int_{t_0}^{t_F}Q_3(s)ds},
    \label{eq:output_sa}
\end{equation}
where $Q_3(\cdot)$ denotes the PDACs concentration for the control parameters set (given by the whole last set of parameters in the Table \ref{tab:par_Opti_1}), $Q_3(\cdot ; \vartheta) $ denotes the PDAC cells concentration for the input parameters $\vartheta$, $t_0 = 10 $ and $t_F = 70$ correspond to the initial time and the finite time for the model simulations. 
The interpretation of the output $\mathcal{V}$ in \eqref{eq:output_sa} is the following:
\begin{itemize}
    \item if $|\mathcal{V}| \leq \epsilon $ for $\epsilon$ arbitrary small, then axons have almost no effect on the appearance of PDACs between days $t_0$ and $t_F$.
    \item If $\mathcal{V} >\epsilon$, the inputs in $\vartheta$ have an inhibiting effect on the PDACs between days $t_0$ and $t_F$.
    \item Conversely, if $\mathcal{V} < - \epsilon $, the inputs in $\vartheta$ have a positive impact on the PDACs between days 10 and 70 which can also be interpreted as a protumoral effect of the axons.
\end{itemize}

\begin{figure}[ht!]
\begin{center}
\includegraphics[width=0.8\textwidth]{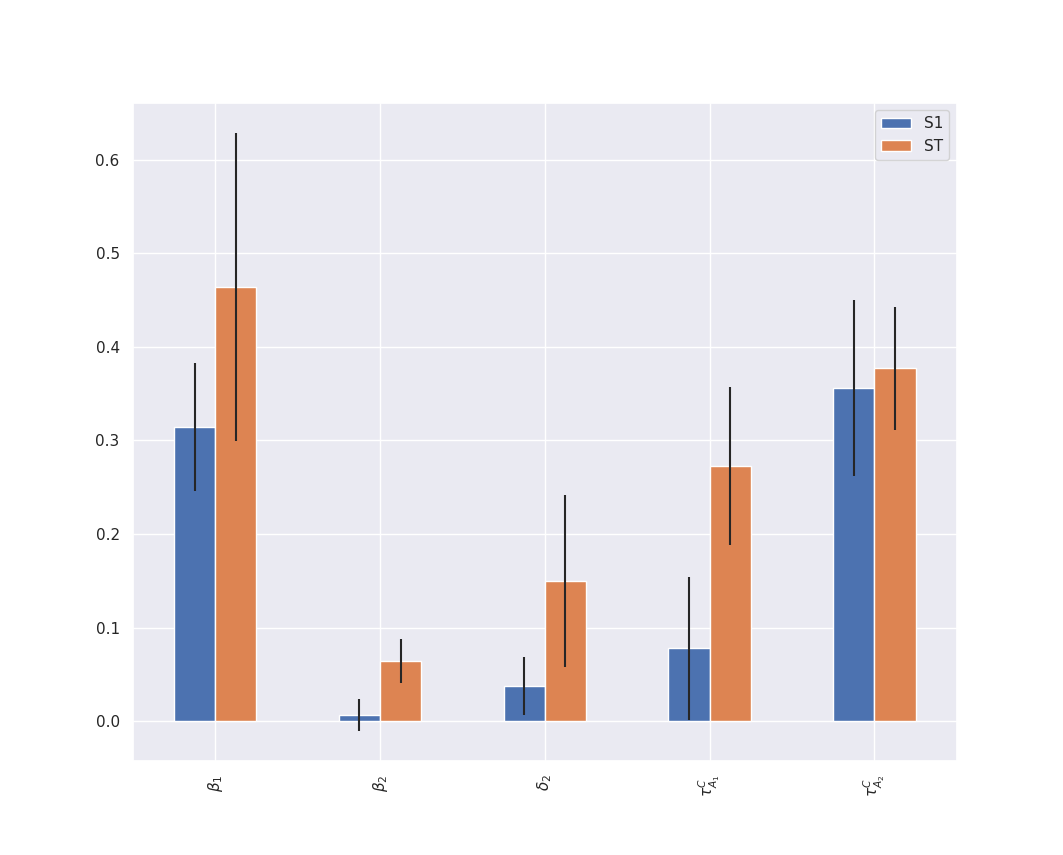}
\end{center}
\caption{\textbf{Indices of the sensitivity analysis of the parameters linked to the axons on the PDAC cells.} {\small The color blue correspond to the first-order-effect indices (S1) of the Sobol sensitivity analysis and the color orange correspond to the total-effect indices (ST). The black segment corresponds to the confidence interval for the associated sensitivity index. }}
\label{fig:sensitivity}
\end{figure}
The results of the variance-based sensitivity analysis are given by the first-order indices and the total-effect indices respectively S1 an ST in the Figure \ref{fig:sensitivity} (see \cite{saltelli2000sensitivity}). Both types of index measure the contribution of the effect $\vartheta_i$ to the output variance. However, the first-order index measure the effect of varying $\theta_i$ alone and is averaged over the variations of the other input parameters. Whereas, the total-effect index gives the contribution of $\vartheta_i$ and its interactions with any other inputs or tuple of inputs. This approach is also called the global sensitivity analysis because it also measures the sensitivity of any tuple of inputs. \\
Ultimately, the following conclusions can be drawn from the sensitivity analysis (see Figure \ref{fig:sensitivity}).
\begin{itemize}
    \item The two main contributors to the variability of the PDACs are the parameters $\beta_1$ and $\tau_{A_2}^C$. Both parameters have an inhibiting effect on the PDACs when their values are growing. However, they do not play the same role in the system (the first regulates the effect of the autonomic axons on the transfer between ADM and PanIN and the latter regulates the effect of the sensory axons on the proliferation). 
    \item By definition, we have that $ST(\vartheta_i) \geq S1(\vartheta_i)$ and the equality holds when the model is additive. One interesting remark is that the first-order index and the total-effect index of $\tau_{A_2}^C$ are almost equivalent. This implies that the contribution of cross-effects between $\tau_{A_2}^C$ and the other parameters is small and thus that the contribution of sensory axons on PDAC proliferation through the proliferation mechanism is almost independent of other axonal mechanisms on the cell populations.
    \item The first-order indices of $ \beta_2, \delta_2 $ and $\tau_{A_1}^C$ are significantly smaller than the indices of the two other parameters. This implies that the variation of these parameters, taken one by one, has a relatively negligible effect on the variations of PDACs. Moreover, since the total-effect index of $\beta_2$ is small, this implies that the inhibitory impact of autonomous axons on the transfer of PanIN cells to PDACs does not have much impact on the overall amount of PDACs in the model.
    \item Concerning $\delta_2$ and $\tau_{A_1}^C$, the discrepancy between the first-order indices and the total-effect indices indicates that the parameters still have an effect on the overall amount of PDACs. However, the significance of the effect is primarily seen through the interactions between one of these two parameters and the others (i.e., when measuring the impact of varying pairs or tuples of parameters simultaneously).
\end{itemize}

\begin{figure}[ht!]
\begin{center}
\includegraphics[width=\textwidth]{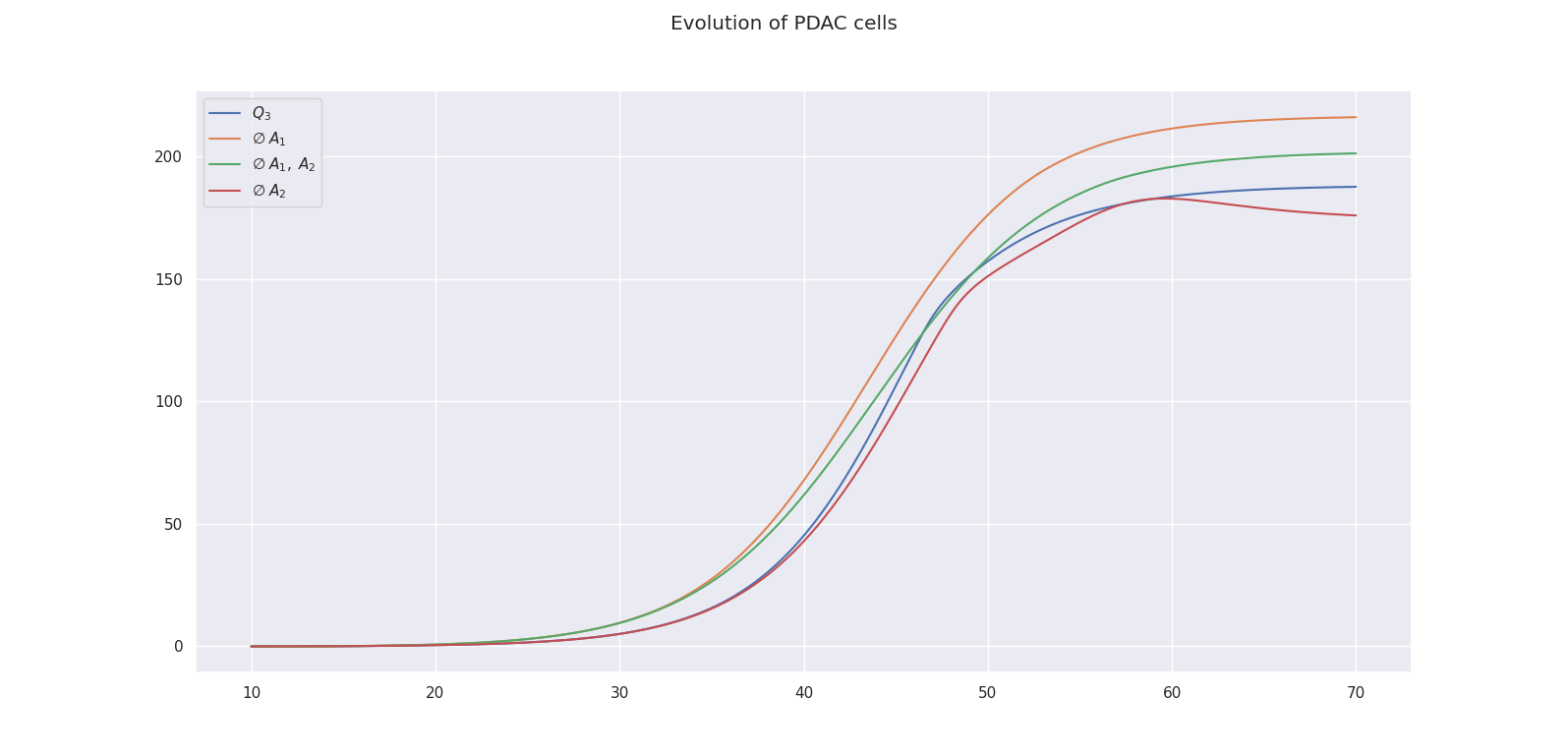}
\end{center}
\caption{\textbf{Numerical simulations of the time evolution of PDAC cells.} {\small The x-axis describes the time in day. Each curves corresponds to the simulation of the quantity $Q_3$ in different cases: when there is no denervation (blue), when the effect of the autonomic axons is off (orange), when the effect of the sensory axons is off (red) and when the effect of the autonomic and sensory axons is off (green). The set of parameters for the simulation is given in the Table \ref{tab:par_Opti_1} (last set of parameters).}}
\label{fig:denerv_curves}
\end{figure}

\textbf{In silico denervation.} A first step of the model validation has been realized in the previous sections through the calibration of the model from the biological data and the study of the identifiability of the parameters. 
In what follows, we consider the dynamics associated to the last set of parameters in Table \ref{tab:par_Opti_1} and define as a control the corresponding evolution of PDACs between days 10 and 70 (the blue curve in Figure \ref{fig:denerv_curves}). 
Now, a more \textit{qualitative} approach of model validation is considered by making denervation in silico. This model validation step is based on three types of denervation and corresponding observations:
\begin{itemize}
    \item in \cite{guillot2020sympathetic} and \cite{renz2018cholinergic} , it has been observed that the denervation of the autonomic axons has a pro-tumoral effect. The pro-tumoral effect of the denervation is also observed in the numerical simulation in Figure \ref{fig:denerv_curves}. The orange curve corresponds to the amount of PDACs ($Q_3 $) when the parameters in front of the autonomic axons in the transfer rates are negligible ($\beta_1 = \beta_2 =0 $) and when the parameters dividing the autonomic axons in the proliferation terms is large $\big(\tau_{A_1}^C = 100$ and $A_1/\tau_{A_1}^C \approx 10^{-3}\big)$. Hence, we see that the PDACs appear earlier and converge toward a bigger plateau compared to the control curve (blue).
    
     \item The denervation of the sensory axons alone corresponding to the red curve in Figure \ref{fig:denerv_curves} shows that the denervation has an anti-tumoral effect by delaying the PDACs arrival. This can be performed by taking $\delta_2=0$ and $\tau_{A_2}^C = 100$ in the model.  The immediate conclusion is that the sensory axons have a pro-tumoral role in the PDAC progression which is corroborated by the experimental conclusions in \cite{saloman2016ablation}.

    \item  In \cite{guillot2020sympathetic} , a pro-tumoral effect has been observed when a denervation of both autonomic and sensory axons has been performed. The same effect is also observed in the numerical simulation (green curve in Figure \ref{fig:denerv_curves}). In order to simulate this additional denervation, we take the same values for the parameters linked to the autonomic axons. We set to 0 the parameter of the sensory axons in the transfer rate ($\delta_2=0$) and to a large value the parameter linked to the effect of the sensory axons in the proliferation terms ($ \tau_{A_2}^C=100$). One can note that compared to the autonomic denervation case, the pro-tumoral effect is not as strong when the sensory axons are also denervated.
   Finally, the three types of denervation performed in vivo are also performed in silico and the in vivo and in silico conclusions are analogous.
   \end{itemize}

}
{
\section{Conclusion}
In this paper, we develop an original model to investigate the role of peripheral axons in pancreatic cancer progression. The study of the calibration to the experimental data highlights the genericity of the
model and the fine analysis of the data informes on the underlying mechanisms. On the one hand, the optimization process to study the parameters identifiability and to obtain the parameters estimations is performed on mouse data but could be adapted to human data. On the other hand, it also gives the intricate links between experimental data, model parameters and underlying biological mechanisms. For
instance, the information on the relatively low amount of PanIN cells coupled to the information on the chronological
time of appearances lead to the conclusion that the speed of tumour progression accelerate towards the
late stage of PDAC development. In addition, the model is very useful for testing hypotheses with the help of numerical
simulations. The model allows us to simulate the effect of a partial or complete denervation at any time.
It sheds lights on complex correlation between the cell populations and the axons and confirms some biological
observations.\\
A first step for expansion and improvement is to investigate further the acquisition of the experimental data in
order to obtain quantitative data reducing the predictive uncertainty of the model. As an example, additional
measurements are needed to quantify more precisely the evolution of sensory and autonomic axon density over
time. These measurements would be related to the speeds of axon proliferation in the model. Moreover, these
speeds aggregate into a coefficient various effects coming from different cell populations. This uncertainty can be reduced by studying and building an optimal experimental design in relation to the mathematical model.\\
The model could be further improved by considering the phenotype of the cell as a continuous variable. In particular, this approach leads to a coupled model with a partial differential equation and differential equations. This formalism would allow a more precise description of the tumour progression and the neuroplastic changes occurring during this process. In particular, this would allow the incorporation of neglected cell categories (e.g. PanIN 1, PanIN 2, PanIN 3, etc) into the model and it would lead to a finer representation of the biological mechanisms.
\\
An additional extension step would be to more accurately include other component of the tumor microenvironment (as example the immune system) in the modeling. However, the precision gained in the modeling automatically leads to an increase in the data required for calibration and an increase in the complexity of the predictions that can be made by this new model. 

}

\medskip 

\textbf{Fundings.}
This research was supported by Centre National de la Recherche Scientifique (CNRS), France; grant from INCa, Fondation Arc et
Ligue contre le cancer (PAIR Pancreas, project title: ”The impact of axonogenesis in pancreatic cancer”, convention number 186738) to FM and FH.

\textbf{Acknowledgements.}
We thank Huyen Thi Trang Nguyen, Adrien Lucchesi, Chloé Dominici and Angélique Puget (Institut de Biologie du Développement de Marseille, IBDM) for the analysis of parasympathetic and sensory innervation in KIC mice. We thank Marie-Jos\'e Chaya for the helpful comments.

\textbf{Competing Interests.} The authors have declared that no competing interests exist.
\bibliographystyle{unsrt}
\bibliography{axons}

\appendix
\section{Additional Lemmas}
\begin{lemma}\label{lem:bnd_y}[Bounds on the cancerous cells]\\
Let $Q_0(0) >0$, $Q_1(0) = Q_2(0) = Q_3(0) =0$, $ A_1(0) \in (-\tau_{A_1}, \tau_{A_1})$ and $A_2(0)\in (0, \tau_{A_2})$. Let $\tau_{A_1} <\tau_{A_1}^C$.
Then it exist $t^* >0$ and two constants $0<c_y<C_y$ such that
$$\forall t> t^*, \quad c_y \leq Q_2(t) +Q_3(t) \leq C_y. $$ 
\end{lemma}
\begin{proof}[Lemma \ref{lem:bnd_y}]\\
We introduce the following notations : $y(t) = Q_2(t) + Q_3(t)$ and $C(A_1(t), A_2(t)) = 1+ \frac{A_1(t)}{\tau_{A_1}^C} + \frac{A_2(t)}{\tau_{A_2}^C}$. 
We recall that 
\begin{align*}
\dt y(t) &= (\gamma_2 Q_2(t) +\gamma_3 Q_3(t))\left(C(A_1(t),A_2(t)) -\frac{y(t)}{\tau_C}\right) + f_1(A_1(t)) Q_1(t). 
\end{align*}
Since $A_2(t) \leq \tau_{A_2}$ and $\tau_{A_1}< \tau_{A_1}^C$ there exists a constant $C_a$ such that $$0<C(A_1(t),A_2(t))<C_a, \quad \forall t>0 .$$
Now, we assume there exists $t_0>0$ such that $$y(t_0) > \tau_C C_a + C $$ where $C>0$ is a constant which will be discussed later and we denote $V_0$ a neighborhood of $t_0$.
Using the bounds on $Q_1$ and $f_1$, we obtain 
\begin{align*}
\dt y(t) &\leq \underline{\gamma} y(t)\left(C_a-\frac{y(t)}{\tau_C}\right) + M_1C(Q_0,Q_1),
\end{align*}
where $\underline{\gamma} \in [\gamma_2, \gamma_3]$. The function $P:\; x\in \mathbb{R}^+ \mapsto \underline{\gamma}x(C_a - x/\tau_C) +M_1C(Q_0,Q_1)$ is polynomial which admits two roots: a negative and a positive one. We denote $y^+$ the positive root of $P$ and we assume $C$ large enough such that $y^+<y(t_0)$.
Then, for $t\in V_0$, we have
$$ \dt y(t) <0 \quad \implies \quad y(t)\leq y(t_0).$$  
Moreover, since the solutions of \eqref{eq:sys_abs} are nonnegative, we obtain a uniform upper bound for $Q_2$ and $Q_3$.\\

Now, we focus on the proof of the lower bound of $y$. First, we prove that it exists $t_1\geq 0$ such that $y(t_1)>0$. Let us assume that $\forall t \geq 0$, $y(t) =0$. It implies that $\dt y(t) =0$ and that $Q_1$ is uniformly equal to $0$. Moreover, $\dt Q_1$ must be equal to 0 and then $\forall t\geq 0 \; Q_0(t) = 0$. It leads to a contradiction since $Q_0(0) >0$.\\
Hence, let us assume that $0<y(t_1)<c_a$ where $c_a = 1 -\frac{\tau_{A_1}}{\tau_{A_1}^C}$ and then $c_a \leq C(A_1(t),A_2(t))$. Moreover, we denote $V_1$ a neighborhood of $t_1$ and we have 
\begin{align*}
\dt y(t) &=(\gamma_2 Q_2(t) +\gamma_3 Q_3(t))\left(C(A_1(t),A_2(t)) -\frac{y(t)}{\tau_C}\right) + f_1(A_1(t)) Q_1(t),\\
& \geq \underline{\gamma}y(t)\left( c_a - \frac{y(t)}{\tau_C}\right).
\end{align*}
Then, for $t\in V_1$, we have $$ \dt y(t) >0 \quad \implies \quad y(t) \geq y(t_1).$$ 

\end{proof}

\begin{lemma}[Study of the vector field]\label{lem:vect_F}
 Let $z=0, \; w=0 $ and $z-\tfrac{f_2(A_1)}{\gamma_2} = 0$ be the three surfaces of interest in order to study the vector field of the system \eqref{eq:sys_abs_lim}.

\begin{itemize}
\item $w= \alpha_2 Q_2 - \alpha_3 Q_3$, $ \overrightarrow{n} = \left(\alpha_2,\; -\alpha_3,\; 0 \right)$ then $$\overrightarrow{n}\cdot f(Q_2,Q_3,A_1) = \alpha_2 Q_2 \left[ z(\gamma_2 - \gamma_3 ) - \left(1 + \frac{\alpha_3}{\alpha_2} \right) f_2(A_1)\right] .$$
\item $z=1 -\frac{Q_2 +Q_3}{\tau_C} + \frac{A_1}{\tau_{A_1}^C} + \frac{\tau_{A_2}}{\tau_{A_2}^C}$,  $ \overrightarrow{n} =\left(-\frac{1}{\tau_C},\; -\frac{1}{\tau_C},\; \frac{1}{\tau_{A_1}^C} \right)$ then $$\overrightarrow{n}\cdot f(Q_2,Q_3,A_1) = \frac{1}{\tau_{A_1}^C}w\left(1 - \left(\frac{A_1}{\tau_{A_1}} \right)^2 \right).$$
\item $z-\frac{f_2(A_1)}{\gamma_2}=1 -\frac{Q_2 +Q_3}{\tau_C} + \frac{A_1}{\tau_{A_1}^C} + \frac{\tau_{A_2}}{\tau_{A_2}^C}-\frac{f_2(A_1)}{\gamma_2}$,  $ \overrightarrow{n} =\left(-\frac{1}{\tau_C},\; -\frac{1}{\tau_C},\; \frac{1}{\tau_{A_1}^C}-\frac{f_2 '(A_1)}{\gamma_2}\right)$ then $$\overrightarrow{n}\cdot f(Q_2,Q_3,A_1) =-\frac{1}{\tau_C}f_2(A_1)\left(\frac{\gamma_3}{\gamma_2} Q_3 +Q_2\right) + \left(\frac{1}{\tau_{A_1}^C}-\frac{f_2 '(A_1)}{\gamma_2}\right)w\left(1 - \left(\frac{A_1}{\tau_{A_1}} \right)^2 \right).$$
\end{itemize}
\end{lemma}

\section{Results on asymptotically autonomous differential systems}\label{app:asymp_aut_sys}
In this section, we recall some results on asymptotically autonomous differential equations. The proofs of the results and further details can be found in \cite{markus2016ii,thieme1994asymptotically,thieme1992convergence}.
\begin{definition}
Let $f:\mathbb{R}\times \mathbb{R}^n \mapsto \mathbb{R}^n$ and $g:\mathbb{R}^n \mapsto \mathbb{R}^n$ be continuous and locally Lipschitz on $\mathbb{R}^n$.
An ordinary differential equation in $\mathbb{R}^n$ 
\begin{equation}
\dot{x}=f(t,x),
\label{eq:asymp_auto}
\end{equation}
is called asymptotically autonomous with limit equation
\begin{equation}
\dot{y}=g(y),
\label{eq:asymp_auto_lim}
\end{equation}
if $$f(t,x)\xrightarrow[t\rightarrow \infty]{} g(x),\qquad \text{locally uniformly in } x\in \mathbb{R}^n.$$
\end{definition}
We denote the $\omega$-limit set of $\omega$ of a forward bounded solution $x$ to \eqref{eq:asymp_auto} satisfying $x(t_0)=x_0 $ by $\omega(t_0,x_0)$:
\begin{equation}
\omega(t_0,x_0)=\bigcap\limits_{s>t_0}\overline{\left\lbrace x(t);\; t\geq s \right\rbrace} .
\label{eq:omeg_set}
\end{equation}

We recall the main theorems established by Markus in \cite{markus2016ii}.
\begin{theorem}
 The $\omega$-limit set $\omega$ of a forward bounded solution $x$ to \eqref{eq:asymp_auto} is nonempty, compact, and connected. Moreover
 $$dist(x(t),\omega) \xrightarrow[t\rightarrow \infty]{} 0. $$
 Finally $\omega$ is invariant under \eqref{eq:asymp_auto_lim}, i.e. if $y(t_0)=y_0 \in \omega $ and $y(t,y_0)$ its trajectory with initial point $y_0$ then $y(t,y_0)\in \omega$. In particular any point in $\omega$ lies on a full orbit of \eqref{eq:asymp_auto_lim} that is contained in $\omega$.
\label{thm:asymp_auto_1}
\end{theorem}

\begin{theorem}
Let $y_\infty$ be a locally asymptotically stable equilibrium of \eqref{eq:asymp_auto_lim} and $\omega$ the $\omega$-limit set of a forward bounded solution $x$ to \eqref{eq:asymp_auto}. 
If $\omega$ contains a point $y_0$ such that the solution of \eqref{eq:asymp_auto_lim} though $(0,y_0)$ converges to $y_\infty$ for $t \to \infty$, then $\omega=\lbrace y_\infty\rbrace$, i.e. $$x(t) \xrightarrow[t\rightarrow \infty]{} y_\infty.$$
\label{thm:asymp_auto_2}
\end{theorem}
These theorems have been used in population dynamics in order to prove that asymptotically autonomous ODEs arising from the models converge to equilibrium (e.g. \cite{castillo1994asymptotically}). 
Moreover, these theorems have been generalized in \cite{thieme1992convergence} to be applied for specific PDEs.

\section{Methods for the identifiability analysis with the profile likelihood.}\label{app:ident_prof_lik}
\subsection{Parameters of the model}
Table \ref{tab:parameters} describes the list of all parameters appearing in the model equations. Each parameter is supplied with its range of values and a short description. The range of parameter values is chosen to be large for two reasons. The first is not to impose too restrictive conditions since the model is completely original and very little information is available. The second is to maximize the search space for reasonable parameters in order to calibrate the model.
\begin{table}[ht!]
\begin{center}
\begin{tabular}{|l|c|c|c|}
  \hline
  Description & Symbol & Units & Range of value  \\
  \hline
  Average density of autonomic axons in healthy pancreas & $A_1^{eq}$ & $nm(\mu m^{-3)}$  & 0.0099\\
  Saturation term of the logistic-like growth & $\tau_{A_1}$ & $nm(\mu m^{-3)}$   & $0.3$\\
  \hline
  Transfer rate of the $Q_0$ to $Q_1$  & $\pi_0$ & day$^{-1}$ & $( 10^{-5},1)$\\
  Transfer rate of the $Q_1$ to $Q_2$  & $\pi_1$ & day$^{-1}$  & $( 10^{-5},1)$\\
  Transfer rate of the $Q_2$ to $Q_3$  & $\pi_2$ & day$^{-1}$  & $( 10^{-5},1)$\\
  \hline
    Growth rate of $Q_2$ & $\gamma_2$ & day$^{-1}$  & $( 10^{-6},1)$\\
  Growth rate of $Q_3$ & $\gamma_3$ & day$^{-1}$  & $( 10^{-6},1)$\\
  \hline
    Amplitude of the effect of $A_1$ on the transfer rate & $\beta_1$ & { $nm^{-1}\mu m^{3}$ }  & $(0,\tau_{A_1}^{-1})$\\
  Amplitude of the effect of $A_1$ on the transfer rate & $\beta_2$ & { $nm^{-1}\mu m^{3}$ } & $(0,\tau_{A_1}^{-1})$\\
  Maximum amplitude of the Michaelis-Menten term & $\delta_0$ & {unitless}  & $(0,10)$\\
  Amplitude of the effect of $A_2$ on the transfer rate & $\delta_2$ & $nm^{-1}$ $\mu m^3$   & $(0,10)$\\
  \hline
  Saturation term of the logistic growth & $\tau_C$ & cells(mm$^{-3})$  & $(50,10^3)$\\
  Saturation term of the logistic growth & $\tau_{A_2}$ & $nm(\mu m^{-3})$   & $(0,2)$\\
  Threshold for the effect of $A_1$ on the growth of $Q_2$ and $Q_3$ & $\tau_{A_1}^C$ & {$nm(\mu m^{-3})$ }  & $(0.3,2)$\\
  Threshold for the effect of $A_2$ on the growth of $Q_2$ and $Q_3$ & $\tau_{A_2}^C$ & {$nm(\mu m^{-3})$ }  & $(0,2)$\\
  \hline
  Amplitude of the effect of $Q_1$ on the growth of $A_1$  & $\alpha_1$ & mm$^3$cell$^{-1}$day$^{-1}{nm(\mu m^{-3)}}$  & $( 10^{-5},1)$\\
  Amplitude of the effect of $Q_2$ on the growth of $A_1$ & $\alpha_2$ & mm$^3$cell$^{-1}$day$^{-1}{nm(\mu m^{-3)}}$   & $( 10^{-5},1)$\\
  Amplitude of the effect of $Q_3$ on the growth of $A_1$ & $\alpha_3$ & mm$^3$cell$^{-1}$day$^{-1}{nm(\mu m^{-3)}}$  & $( 10^{-5},1)$\\
  Amplitude of the effect of $Q_2$ on the growth of $A_2$ & $\bar{\alpha}_2$ & mm$^3$cell$^{-1}$day$^{-1}$   & $( 10^{-5},1)$\\
  Amplitude of the effect of $Q_3$ on the growth of $A_2$ & $\bar{\alpha}_3$ & mm$^3$cell$^{-1}$day$^{-1}$  & $( 10^{-5},1)$\\
  \hline
\end{tabular}
\end{center}
\caption{\textbf{List of the parameters for the model described by the equations \eqref{eq:Mod_Q0}-\eqref{eq:Mod_A2}.} { Recall that units associated to  the $Q_i$ are cells$(mm^{-3})$ and to $A_i$ $nm(\mu m^{-3})$}.
}
\label{tab:parameters}
\end{table}
It is then possible to divide the parameters of the model to estimate into five categories:
\begin{enumerate}
\item \textbf{The transfer rates $\pi_0, \; \pi_1  $ and $\pi_2 $.} These parameters describe the main rates of cell transfer. More specifically, the time evolution of cell and axon populations is simulated by a compartmental model and these parameters quantify the speed of transfer from one to another compartment.
\item \textbf{The proliferation speeds $\gamma_2 $ and $\gamma_3$.} These parameters give the growth speeds of the PanIN and PDAC cell populations. 
\item \textbf{The parameters of the regulations on the transfer rates $\beta_1 $, $\beta_2 $, $\delta_0 $ and $\delta_2 $.} These parameters are multiplicative coefficients  appearing in the transfer rates. They regulate either positively or negatively the speed of transfer from one compartment to another in the model.
\item \textbf{The saturation rates $\tau_C $, $\tau_{A_1}^C $, $\tau_{A_2} $ and $\tau_{A_2}^C $.} The saturation rates are closely linked to bio-physical constraints such as the maximal volume of the model's domain (i.e. the pancreas) and the maximal axons densities in the domain. These two quantities might vary from one individual to another, however one can reasonably assume maximal bounds and implement it in the model.
\item \textbf{The parameters of the regulations on the proliferation speeds of axons $\alpha_1, $ $\alpha_2, $ $\alpha_3, $ $\bar{\alpha}_2 $ and $\bar{\alpha}_3 $.} These parameters are multiplicative coefficients appearing in the growth terms of the axons. Their main effect is to modulate the rate of axon proliferation depending on the amount of cell populations present at the observed time.
\end{enumerate}
Moreover, the density of autonomic axons in a healthy pancreas $A^{eq}_1$ is given by the control experiments and amounts to 0.980 nm($\mu$m$^{-3})$. The other parameter set in this model is $\tau_{A_1}$ (i.e. the saturation rate of $A_1$). Since $A_1$ is the variation of the density of autonomic axons with respect to $A^{eq}_1$, setting the value of $\tau_{A_1}$ to 0.3 implies that the maximum variation of the density of autonomic axons does not exceed one third of its equilibrium state. This assumption can be justified in the model by the fact that natural innervation and denervation in vivo appear to be phenomena at the margin during the development of the pancreatic adenocarcinoma. In addition, autonomic axons appear to be pushed to the periphery of the organ during the development of PDAC cells
and the domain of the model is the pancreas and its immediate surroundings.

\begin{table}[ht!]
    \centering
    \tiny{
\begin{tabular}{rrrrrrrrrrrrrrrrrr}
 $\pi_0 $ & $\delta_0 $ & $\pi_1 $ & $\beta_1 $ & $\gamma_2 $ & $\pi_2 $ & $\beta_2 $ & $\delta_2 $ & $\gamma_3 $ & $\tau_C$ & $\alpha_1 $ & $\alpha_2 $ & $\alpha_3 $ & $\tau_{A_1}^C $ & $\bar{\alpha}_2$ & $\bar{\alpha}_3$ & $\tau_{A_2} $ & $\tau_{A_2}^c $\\
\hline
 1.5e-3 & 8.4 & 1.9e-5 & 3.2 & 4.2e-1 & 1.1e-3 & 2.2 & 3.2 & 9.3e-1 & 8.5e+1 & 1.6e-5 & 8.4e-2 & 4.7e-4 & 3.6e-1 & 6.0e-1 & 1.6e-5 & 1.9 & 1.6\\
 1.5e-3 & 7.8 & 6.4e-3 & 1.5e-1 & 1.2e-6 & 8.6e-2 & 6.0e-1 & 8.6 & 4.3e-1 & 1.0e+2 & 1.7e-3 & 2.2e-1 & 1.5e-3 & 1.8 & 7.2e-1 & 1.4e-5 & 2.0 & 1.4\\
 1.3e-3 & 8.8 & 8.6e-3 & 2.9e-1 & 3.6e-1 & 2.7e-1 & 2.5e-1 & 2.6e-1 & 3.6e-1 & 1.0e+2 & 4.4e-5 & 1.3e-2 & 7.8e-5 & 2.0 & 3.5e-1 & 1.5e-5 & 1.9 & 1.6\\
 1.3e-3 & 9.9 & 1.8e-5 & 3.3 & 6.9e-1 & 1.6e-1 & 3.3e-2 & 2.0e-1 & 6.9e-1 & 6.3e+1 & 3.4e-4 & 7.5e-3 & 8.6e-5 & 2.0 & 3.1e-1 & 3.8e-5 & 2.0 & 1.5e-1\\
 1.4e-3 & 9.1 & 4.5e-3 & 6.5e-2 & 8.2e-2 & 5.5e-2 & 9.9e-2 & 9.7 & 4.3e-1 & 9.3e+1 & 1.2e-5 & 5.3e-2 & 3.0e-4 & 5.4e-1 & 6.6e-1 & 1.3e-5 & 1.7 & 2.0\\
 1.7e-3 & 7.7 & 9.1e-3 & 3.2 & 2.1e-6 & 7.5e-2 & 1.5e-1 & 8.9 & 3.7e-1 & 1.0e+2 & 1.9e-5 & 3.8e-2 & 2.2e-4 & 7.7e-1 & 6.5e-1 & 2.3e-5 & 1.8 & 1.6\\
 2.2e-3 & 4.9 & 5.6e-2 & 3.2 & 1.0e-2 & 4.9e-1 & 3.3 & 3.1e-1 & 2.0e-1 & 1.9e+2 & 1.0e-5 & 4.8e-1 & 6.0e-3 & 2.0 & 9.6e-1 & 1.0e-5 & 1.5e-1 & 2.0
 
\end{tabular}
}
    \caption{Table of parameters for the numerical computations in Figure \ref{fig:Opti_1}}
    \label{tab:par_Opti_1}
\end{table}

\subsection{Numerical method to assess the identifiability with the profile likelihood}
In this section, we detail the method used to study the identifiability of the parameters. It is assimilated to an  optimization problem and can be considered as a \textit{one-at-a-time} identifiability process since we study the profile likelihood of each parameter independently. We recall that the parameters are denoted by the vector $\theta$ which include all the parameters except for $\tau_{A_1} $ and $A_1^{eq}$ fixed as in Table \ref{tab:parameters}. The computation of the profile likelihood is broken down as follows:\\

\textit{Step 1.} For each component $\theta_j$ of the vector of parameters, we choose a sequence of 20 values which discretizes its range of values. The sequence discretize uniformly the range of values of the parameters $ \delta_0, \; \beta_1, \; \beta_2 , \; \delta_2 , \;\tau_C , \; \tau_{A_1}^C, \; \tau_{A_2} $ and $\tau_{A_2}^C$. For the other parameters, the sequence discretize uniformly the log$_{10}$ transformation of their respective ranges of value. This allows us to explore more precisely the optimization domain by taking into account the differences in the orders of magnitude between parameters. The sequence of values chosen for the parameter $\theta_j$ is denoted $\left(p_i^j\right)_{i=1\hdots 20}. $\\

\textit{Step 2.} For each $p_i^j$, we minimize $\Chi_2$ defined by \eqref{eq:cost_func_disc} on ${\theta \in \left\lbrace \theta\;  |\; \theta_j =p_i^j \right\rbrace}$. Since, the problem is non-linear, non-convex and high dimensional, we choose to repeat 50 times the optimisation procedure using as the initialization step a random sets of parameters chosen by the uniform distribution over their respective value range (or over the log$_{10}$ transformation of their value range). For this optimization problem, we use the CMA evolution strategy algorithm \cite{hansen2016cma} : at each step of the optimization loop, the algorithm picks a set of parameters given by a specific random distribution over the optimization domain. It evaluates the objective function value for this set of parameters and iterate by updating the random distribution with the barycenter of the "best" parameters (i.e. whose objective function values are the lowest). The algorithm stops if the distance between the best evaluation of the objective function of the last 40 iterations of the optimization loop and all the objective function values of the last iteration is less than the tolerance threshold $10^{-3}$.\\

\textit{Step 3.} For each $\theta_j$, we plot the first quartile, the last quartile and the median of the 50 optimal objective function values obtained in order to visualize the results of \textit{step 2}. It gives the following Figures \ref{fig:boxplot_pl1}-\ref{fig:boxplot_pl2}. If the curves formed by the median describe a convex shape and a minimum is clearly attained in a distinct value, we consider that the parameter is identifiable and its estimate is this distinct values (for instance, in Figure \ref{fig:boxplot_pl1}, the top panel shows that median attains its minimal value at $\pi_0 = 0.0025.$). Else, we iterate \textit{step 2} once or twice on a reduced domain where the already identifiable parameters are fixed to their estimated values.\\

The numerical results of this optimization process are summarized in Figures \ref{fig:boxplot_pl1}-\ref{fig:boxplot_pl2}. We denote \textit{trial 1} the first iteration of \textit{step 2}, \textit{trial 2} the second and \textit{trial 3} the third. In total, seven parameters are identifiable (see Figure \ref{fig:ident_graph} and Table \ref{tab:ident_par}). We choose to stop iterating the identifiability process after the third trial considering the amount of data at our disposal and the over fitting issues coming from the high dimensionality of our problem.
 One can reasonably justify this choice by looking at the distribution of the objective function values in \textit{trial 3} (see Figures \ref{fig:boxplot_pl1}-\ref{fig:boxplot_pl2}). The cost function values over the whole range for the remaining parameters are concentrated under 10 and the median of the cost distribution does not have a clearly distinguishable minimum value. An extra iteration might not ensure a sufficient difference between the minimum of the median and the rest of the median values. \\
\begin{figure}[!ht]
\centering
\includegraphics[width=0.8\textwidth]{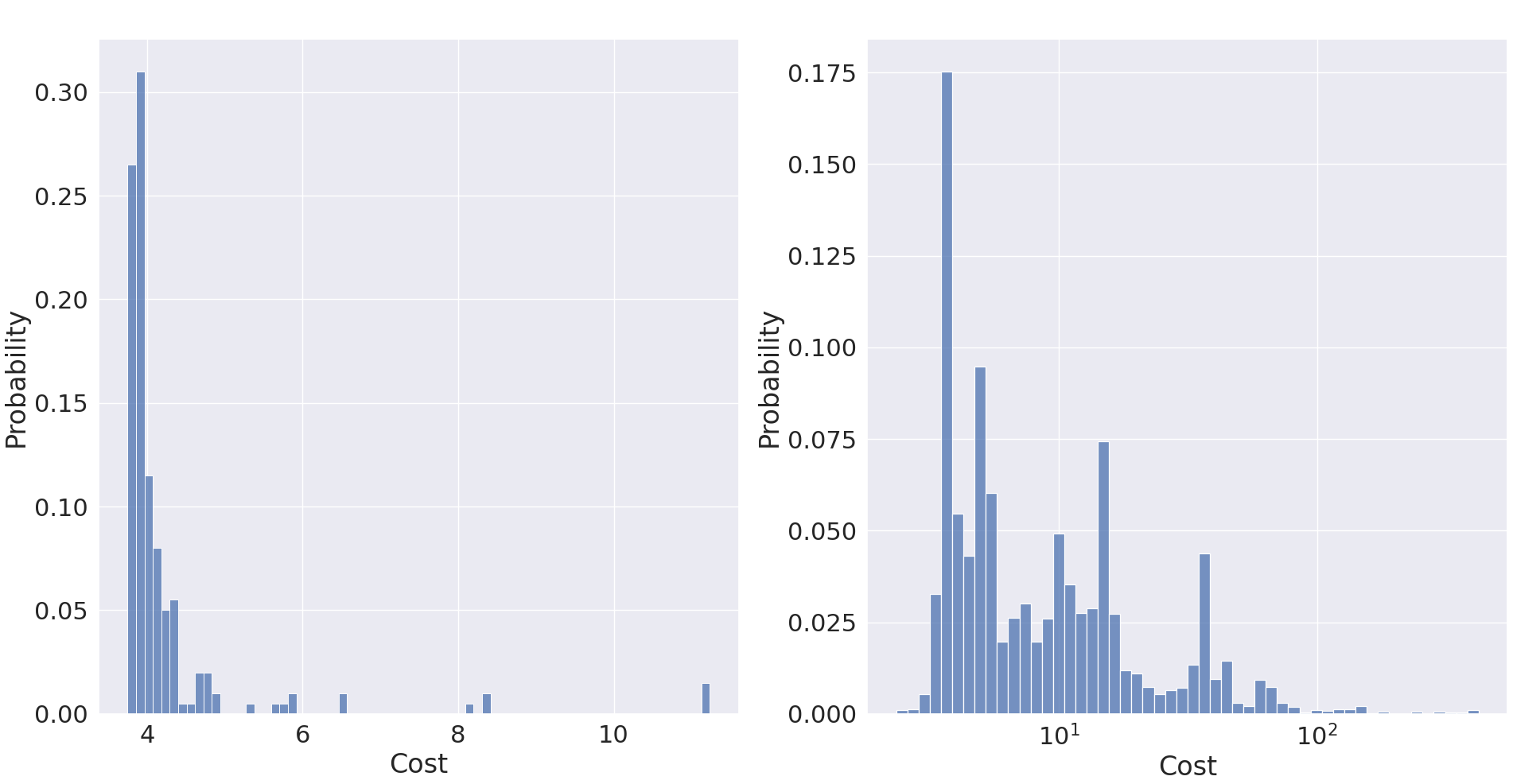}
\caption{\textbf{Distribution of the evaluations of the optimal objective function.} The x-axis is the range of the cost function. The y-axis is the frequency over a sample of size 200. }
\label{fig:hist_ident}
\end{figure}
 As a final validation step, we perform the following numerical experiments:
we fix the seven identifiable parameters to their values (see Table \ref{tab:ident_par}), we pick 200 sets of rescaled parameters (for the non identifiable ones) from a uniform random distribution over their value ranges and we compute the minimization problem 
$$C = \min\limits_{ \theta \; | \; \theta_j = p^{j,*}, \; j \in J  }=  \Chi_2 (\theta) $$
where $\Chi_2$ is defined by \eqref{eq:cost_func_disc} and $J$ is the set of coordinates of the indentifiable parameters denoted $p^{j,*}$. It follows that the distribution of the optimal objective function evaluations is concentrated on 3 (see Figure \ref{fig:hist_ident}). For instance, the costs of the trajectories shown in Figure \ref{fig:Opti_1} are superior or equal to 3. This numerical experiments ensures that the results on the identifiable parameters are satisfactory from a qualitative point of view.
\begin{figure}[!ht]
\centering
        \includegraphics[width=0.9\linewidth]{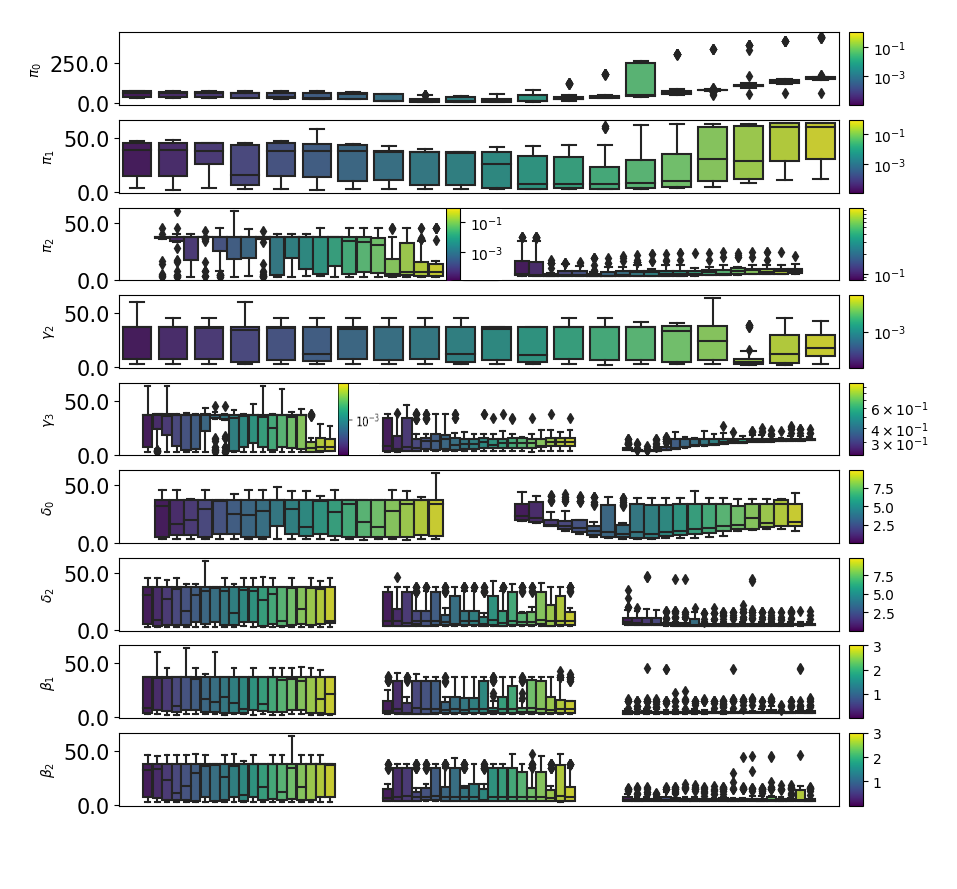}
        \caption{\textbf{Numerical computation of the profile likelihood for nine parameters.} \small{Each row corresponds to a parameter. This parameter takes twenty distinct values represented by the color gradient. The numerical method gives a cost distribution for each value of the parameter. This distribution is then visualized by a boxplot (the vertical axis represents the cost). This method is iterative for the non-identifiable parameters and this leads to a new simulation on a reduced optimization domain. For instance, $\pi_0$, $\pi_1$ and $\gamma_2$ are identifiable from the first iteration. The parameters $\pi_2$ and $\delta_0$ are identifiable from the second iteration and $\gamma_3$ from the third iteration. The other parameters are non-identifiable after three iteration.}  }
        \label{fig:boxplot_pl1}
\end{figure}
\begin{figure}[!ht]
\centering
        \includegraphics[width=0.9\linewidth]{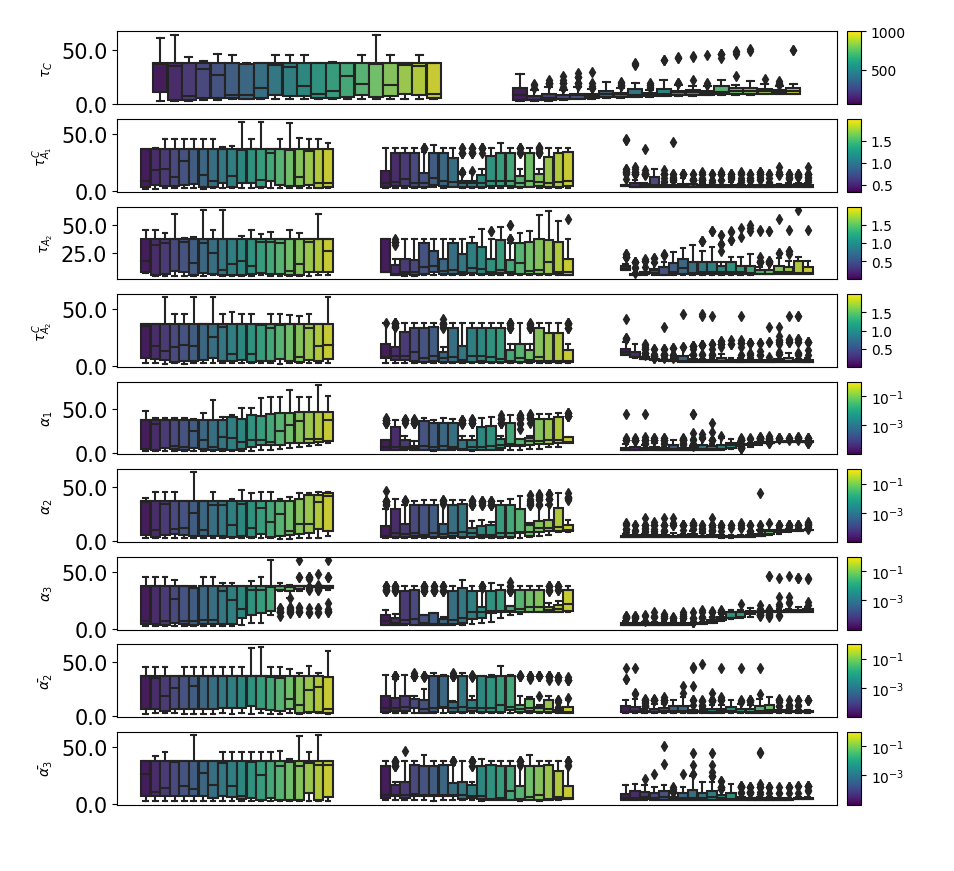}
        \caption{\textbf{Numerical computation of the profile likelihood for the other parameters.} \small{Each row corresponds to a parameter. This parameter takes twenty distinct values represented by the color gradient. The numerical method gives a cost distribution for each value of the parameter. This distribution is then visualized by a boxplot (the vertical axis represents the cost). This method is iterative for the non-identifiable parameters and this leads to a new simulation on a reduced optimization domain. For instance, only $\tau_C$ is identifiable (from the second iteration). The other parameters are non-identifiable.}}
        \label{fig:boxplot_pl2}
\end{figure}
\end{document}